\documentclass[a4paper]{article}

\usepackage{mathtools}
\usepackage{amsmath,amssymb,amsthm,mathrsfs,enumitem,mathtools}
\usepackage{cite}
\usepackage{upgreek}
\usepackage[hypcap=false]{caption}
\usepackage{subcaption}
\usepackage{ifthen}
\usepackage{array}
\usepackage{booktabs}
\usepackage{tabularx}

\usepackage[pdftex,bookmarksopen=true,bookmarks=true,unicode,setpagesize]{hyperref}
\hypersetup{colorlinks=true,linkcolor=black,citecolor=black}
\usepackage{pgf,tikz}
\usetikzlibrary{arrows,patterns}

\let\equation=\gather
\let\endequation=\endgather

\allowdisplaybreaks[3]
\numberwithin{equation}{section}

\makeatletter
\renewcommand*{\@fnsymbol}[1]{\ensuremath{\ifcase#1\or 1\or 2\or
   3\else\@ctrerr\fi}}
\makeatother

\newtheorem{theorem}{Theorem}[section]
\newtheorem{corollary}[theorem]{Corollary}
\newtheorem{lemma}[theorem]{Lemma}
\newtheorem{proposition}[theorem]{Proposition}
\theoremstyle{definition}
\newtheorem{definition}[theorem]{Definition}
\newtheorem{example}[theorem]{Example}
\theoremstyle{remark}
\newtheorem{remark}[theorem]{Remark}

\newcounter{assum}
\newenvironment{assum}[1][]{\ifx\newenvironment#1\newenvironment\refstepcounter{assum}\fi\equation\tag{\ensuremath{\mathrm{A}\theassum#1}}}{\endequation}

\renewcommand{\L}{\mathfrak{L}}
\renewcommand{\Re}{\mathrm{Re}\,}
\renewcommand{\Im}{\mathrm{Im}\,}

\newcommand{\R}{\mathbb{R}}
\newcommand{\X}{{\mathbb{R}^d}} 
\newcommand{\N}{\mathbb{N}}
\newcommand{\eps}{\varepsilon}
\newcommand{\la}{\lambda}

\newcommand{\ka}{\varkappa}

\newcommand{\kl}{\varkappa_{{\ell}}}
\newcommand{\kn}{\varkappa_{{n\ell}}}
\newcommand{\kam}{\varkappa^-}
\newcommand{\kap}{\varkappa^+}
\newcommand{\kapm}{\varkappa^\pm}
\newcommand{\1}{1\!\!1}
\newcommand{\m}{{\mathfrak{m}}}
\newcommand{\A}{{\mathfrak{a}}}
\newcommand{\h}{{\mathfrak{h}}}
\newcommand{\T}{\mathfrak{t}}
\newcommand{\An}{\mathcal{H}}
\newcommand{\M}{{\mathcal{M}_\theta}(\R)} 
\newcommand{\Vxi}{\mathcal{V}_\xi}
\newcommand{\Wxi}{\mathcal{W}_\xi}
\newcommand{\Uxi}{\mathcal{U}_\xi}
\newcommand{\aplus}{\boldsymbol{a}^+}
\newcommand{\aminus}{\boldsymbol{a}^-}
\newcommand{\aplusminus}{\boldsymbol{a}^\pm}
\newcommand{\newaplus}{a^+}
\newcommand{\newaminus}{a^-}
\newcommand{\newaplusminus}{a^\pm}
\newcommand{\sigmaplus}{\widehat{\sigma}}

 \title{Doubly nonlocal {F}isher--{KPP} equation: Speeds~and~uniqueness of traveling waves}

\author{Dmitri Finkelshtein\thanks{Department of Mathematics,
Swansea University, Singleton Park, Swansea SA2 8PP, U.K. ({\tt d.l.finkelshtein@swansea.ac.uk}).} \and Yuri Kondratiev\thanks{Fakult\"{a}t
f\"{u}r Mathematik, Universit\"{a}t Bielefeld, Postfach 110 131, 33501 Bielefeld,
Germany ({\tt kondrat@math.uni-bielefeld.de}).} \and Pasha Tkachov\thanks{Gran Sasso Science Institute, Viale Francesco Crispi, 7, 67100 L'Aquila AQ, Italy ({\tt pasha.tkachov@gssi.it}).}}

\begin{document}

\maketitle

\begin{abstract}
We study traveling waves for a reaction-diffusion equation with nonlocal anisotropic diffusion and a linear combination of local and nonlocal monostable-type reactions. We describe relations between speeds and asymptotic of profiles of traveling waves, and prove the uniqueness of the profiles up to shifts.



\end{abstract}

\section{Introduction} 
We will study traveling wave solutions to the equation
\begin{equation}
\begin{aligned}
\dfrac{\partial u}{\partial t}(x,t)&=\kap \int_{\X }\aplus  (x-y)u(y,t)dy-m u(x,t)-u(x,t)\, G\bigl(u(x,t)\bigr),\\
 G\bigl(u(x,t)\bigr)&:= \kl u(x,t) + \kn \int_{\X }\aminus (x-y)u(y,t)dy.
\end{aligned}
\label{eq:basic}
\end{equation}
Here $d\in\N$; $\kap, m>0$ and $\kl, \kn\geq0$  are constants, such that 
\begin{equation}\label{eq:nondegencomp}
      \kam := \kl+\kn > 0;
\end{equation}    
the kernels  $0\leq \aplusminus \in L^{1}(\X)$ are probability densities, i.e. $\int_{\X }\aplusminus (y)dy=1$.

For the case of the local nonlinearity in \eqref{eq:basic}, when $\kn=0$, the equation \eqref{eq:basic}  was considered, in particular, in \cite{CD2007,Yag2009,AGT2012,Gar2011,LSW2010,SLW2011,CDM2008,Sch1980,ZLW2012,SZ2010}. For a nonlocal nonlinearity and, especially, for the case $\kl=0$ in \eqref{eq:basic}, see e.g. \cite{Dur1988,FM2004,FKK2011a,FKKozK2014,PS2005,FKMT2017,FT2017c,YY2013}. For details, see the introduction to \cite{FKT100-1} and also the comments below.

Under assumption
\begin{assum}\label{as:chiplus_gr_m}
\kap >m,
\end{assum}
the equation \eqref{eq:basic} has two constant stationary solutions: $u\equiv0$ and $u\equiv\theta$, where
\begin{equation}\label{theta_def}
  \theta:=\frac{\kap -m}{\kam }>0.
\end{equation}
Moreover, one can then also rewrite the equation in a reaction-diffusion form
\[
  \dfrac{\partial u}{\partial t}(x,t)=\kap \int_{\X }\aplus  (x-y)\bigl( u(y,t)-u(x,t)\bigr) dy+u(x,t)\Bigl(\beta - G\bigl(u(x,t)\bigr)\Bigr),
\]
where $\beta=\kap-m>0$. We treat then \eqref{eq:basic} as a doubly nonlocal Fisher--KPP equation, see the introduction to \cite{FKT100-1} for details.

By a (monotone) traveling wave solution to \eqref{eq:basic} in a direction $\xi\in S^{d-1}$ (the unit sphere in $\X$), we will understand a solution of the form 
\begin{equation}\label{eq:deftrw}
      \begin{gathered}
      u(x,t)=\psi(x\cdot\xi-ct),  \quad t\geq0, \ \mathrm{a.a.}\ x\in\X, \\
      \psi(-\infty)=\theta, \qquad \psi(+\infty)=0,
      \end{gathered}
\end{equation}
where $c\in\R$ is called the speed of the wave and the function $\psi\in\M$ is called the profile of the wave. 
Here $\M$ denotes the set of all decreasing and~right-continuous functions $f:\R\to[0,\theta]$, and $x\cdot \xi$ denotes the scalar product in~$\X$.

The present paper is a continuation of \cite{FKT100-1}; they both are based on our unpublished preprint \cite{FKT2015} and thesis \cite{Tka2017}. In~\cite[Propositions 3.7]{FKT100-1}, we have shown (cf.~also~\cite{CDM2008}) that the study of a traveling wave solution \eqref{eq:deftrw} with a fixed $\xi\in S^{d-1} $ can be reduced to the study of the one-dimensional version of \eqref{eq:basic} with the kernels
\begin{equation}\label{apm1dim}
     \newaplusminus (s):=\int_{\{\xi\}^\bot} \aplusminus (s\xi+\eta) \, d\eta, \quad s\in\R,
\end{equation}
where $\{\xi\}^\bot:=\{x\in\X\mid x\cdot\xi=0\}$. For $d=1$ and $\xi\in S^0=\{-1,1\}$, \eqref{apm1dim} reads as follows: $ \newaplusminus (s)=\aplusminus (s\xi)$, $s\in\R$. Clearly, $\int_\R  \newaplusminus (s)\,ds=1$.

For simplicity, we omit $\xi$ from the notations for functions $\newaplusminus$, assuming that the direction $\xi\in S^{d-1} $ is fixed for the sequel. We denote also 
\begin{equation}\label{eq:defJtheta}
    J_\theta(s):=\kap \newaplus (s)-\theta\kn\newaminus (s), \quad s\in\R.
\end{equation}

Under \eqref{as:chiplus_gr_m}, we assume that
\begin{assum}
\label{as:aplus_gr_aminus}
   J_\theta \geq 0,\quad \text{a.a.}\ s\in\R,
\end{assum}
and that there exists $\mu=\mu(\xi)>0$, such that
\begin{assum}
\label{aplusexpint1}
  \int_\X \aplus (x) e^{\mu x\cdot \xi}\,dx= \int_\R \newaplus (s) e^{\mu s}\,ds<\infty.
\end{assum}
Stress that the assumption \eqref{as:aplus_gr_aminus} is redundant for the case of the local nonlinearity in \eqref{eq:basic}, when $\kn=0$.

Suppose also, that $\aplus$ is not degenerated in the direction $\xi\in S^{d-1} $, i.e.
\begin{assum}\label{as:a+nodeg-mod}
    \begin{gathered}
    \text{there exist $r=r(\xi)\geq0$, $\rho=\rho(\xi)>0$, $\delta=\delta(\xi)>0$, such that}\\
    \newaplus (s)\geq\rho, \text{ for a.a.\! $s\in [r- \delta, r+ \delta]$.}
    \end{gathered}
  \end{assum}
A sufficient condition for \eqref{as:a+nodeg-mod} is that $\aplus(x)\geq\rho'$ for a.a. $x\in\X$ such that $|x-r\xi|\leq \delta'$, for some $\rho',\delta'>0$.

\begin{theorem}[\!\!{\cite[Theorem 1.1, Propositions 3.7, 3.14, 3.15]{FKT100-1}}]\label{thm:trwexists}
Let $\xi\in S^{d-1} $ be fixed and suppose that \eqref{as:chiplus_gr_m}, \eqref{as:aplus_gr_aminus}, \eqref{aplusexpint1} hold. Then there exists $c_*=c_*(\xi)\in\R$, such that, for any $c<c_*$,  a traveling wave solution to \eqref{eq:basic} of the form \eqref{eq:deftrw} with $\psi\in\M$ does not exist; whereas for any $c\geq c_*$,
\begin{enumerate}[label={\arabic*})]
          \item there exists a traveling wave solution to \eqref{eq:basic} with the speed $c$ and a profile $\psi\in\M$ such that  \eqref{eq:deftrw} holds;
          \item if $c\neq0$, then the profile $\psi\in C_b^\infty(\R)$ (the class of infinitely many times differentiable functions on $\R$ with bounded derivatives); if $c=0$ (in the case $c_*\leq 0$), then $\psi\in C(\R)$;
          \item there exists $\mu=\mu( c, \newaplus , \kam, \theta)>0$ such that
              \begin{equation}\label{eq:trwexpint}
                \int_\R\psi(s)e^{\mu s}\,ds<\infty;
              \end{equation}    
          \item if, additionally, \eqref{as:a+nodeg-mod} holds, then, for any $c\neq0$, there exists $\nu>0$, such that $\psi(t)e^{\nu t}$ is~a~strictly increasing function;
          \item if, additionally, \eqref{as:a+nodeg-mod} holds with $r=0$, then the profile $\psi$ is a strictly decreasing function on $\R$.
        \end{enumerate}
\end{theorem}

The smoothness of the profile $\psi$ implies, see \cite[Proposition 3.11]{FKT100-1} for details, that $\psi$ satisfies the equation
\begin{equation}
            c\psi'(s) +\kap (\newaplus*\psi)(s) -m\psi(s)
              -\kn\psi(s) (\newaminus*\psi)(s) -\kl \psi^2(s)=0 \label{eq:trw}
\end{equation}  
for all $s\in\R$. Here $*$ denotes the classical convolution of functions on $\R$, i.e. 
\[
  (\newaplusminus*\psi)(s):=\int_\R  \newaplusminus (s-\tau)\psi(\tau)\,d\tau , \quad s\in\R.
\] 

To study \eqref{eq:trw}, we will use a bilateral-type Laplace transform
\begin{equation}\label{Laplace}
(\L f)(z)=\int_\R f(s)e^{z s}\,ds,\quad \Re z>0, \ f\in L^\infty(\R).
\end{equation}
For each $f\in L^\infty(\R)$, there exists $\sigma(f)\in[0,\infty]$, called the abscissa of $f$, such that the integral in \eqref{Laplace} is convergent for $0<\Re z<\sigma(f)$ and divergent for $\Re z>\sigma(f)$, see Lemma~\ref{lem:allaboutLapl} below for details.

We assume that 
\begin{assum}\label{boundedkernels}
    \newaplus \in L^\infty(\R),
\end{assum}
that is evidently fulfilled if e.g. $\aplus \in L^\infty(\X)$. Then, under \eqref{aplusexpint1} and \eqref{boundedkernels}, there exists $\sigma(\newaplus )\in(0,\infty]$. Similarly, 
because of \eqref{eq:trwexpint}, for any profile $\psi$ of a traveling wave solution to \eqref{eq:basic}, there exists 
$\sigma(\psi)\in(0,\infty]$.

Finally, for the fixed $\xi\in S^{d-1}  $, we assume that
  \begin{assum}\label{firstmoment}
   \int_\X \lvert x\cdot \xi \rvert \, \aplus (x)\,dx=\int_\R \lvert s \rvert \newaplus(s)\,ds <\infty.
  \end{assum}
  Under assumption \eqref{firstmoment}, we define
\begin{equation}\label{firstdirmoment}
\m_\xi:=\int_\X x\cdot \xi \, \, \aplus (x)\,dx=\int_\R s\, \newaplus(s)\,ds.
\end{equation}

We formulate now the first main result of the present paper.
\begin{theorem}\label{thm:trwcandpsi}
Let, for the fixed $\xi\in S^{d-1} $, the conditions \eqref{as:chiplus_gr_m}--\eqref{firstmoment} hold. Let $c_*=c_*(\xi)\in\R$ be the minimal traveling wave speed according to Theorem~\ref{thm:trwexists}, and let, for any $c\geq c_*$, the function $\psi=\psi_c\in\M$ be a traveling wave profile corresponding to the speed~$c$.  Then
\begin{enumerate}
  \item There exists a unique $\la_*\in\R$, such that
\begin{equation}
 \begin{aligned}
      c_*&=\min_{\la>0}\frac{1}{\la}\biggl(\kap \int_\R \newaplus(s)e^{\la s}\,ds -m\biggr)\\&=
\frac{1}{\la_*}\biggl(\kap \int_\R \newaplus(s)e^{\la_* s}\,ds -m\biggr)>\kap \m_\xi.
 \end{aligned}\label{minspeed}
\end{equation}
\item For any $c\geq c_*$ the abscissa of the corresponding profile $\psi_c$ is finite:
\begin{equation}
    \label{finitespeed}
\sigma(\psi_c)\in(0,\la_*],
\end{equation}
and the mapping $(0,\la_*]\ni \sigma(\psi_c) \mapsto c\in[c_*,\infty)$ is a (strictly) decreasing bijection, given by
\begin{equation}\label{speedviaabs}
c=\frac{1}{\sigma(\psi_c)}\biggl(\kap \int_\R \newaplus(s)e^{\sigma(\psi_c)\, s}\,ds -m\biggr).
\end{equation}
In particular, 
\begin{equation}\label{eq:argminofGisabs}
    \sigma(\psi_{c_*})=\la_*.
\end{equation}
\item For any $c\geq c_*$,
\begin{equation}\label{Laplinfatabs}
(\L \psi_c) \bigl(\sigma(\psi_c)\bigr) =\infty.
\end{equation}
\end{enumerate}
\end{theorem}

Note that, in light of \eqref{minspeed}, the kernel $\aplus$ may be so \emph{slanted} to the direction opposite to $\xi$, that $c_*(\xi)<0$. A sufficient condition to exclude this, by the inequality in \eqref{minspeed}, is that $\m_\xi=0$; in particular, this evidently holds if $\newaplus$ is symmetric.

We will show also that $\sigma(\psi_{c_*})=\la_*\leq \sigma(\newaplus)$.
We will distinguish two cases: the non-critical case when $\sigma(\psi_{c_*})< \sigma(\newaplus)$, and the critical case when $\sigma(\psi_{c_*})= \sigma(\newaplus)$.
Note that a kernel $\newaplus$ which is compactly supported or decreases faster~than any exponential function corresponds to the non-critical case, as then  $\la_*<\infty=\sigma(\newaplus)$.

The critical case is characterized by the following conditions (cf. Proposition~\ref{prop:Gnegholds} and Definition~\ref{def:VWxi} below)
\begin{gather}
    \sigmaplus :=\sigma(\newaplus )<\infty, \qquad \int_\R (1+|s|)\newaplus (s)e^{\sigmaplus  s}\,ds<\infty,\label{eq:critcase}\\
    m\leq \kap \int_\R (1-\sigmaplus  s)\newaplus (s)e^{\sigmaplus  s}\,ds. \label{eq:ineqform}
\end{gather}
Note that, informally, \eqref{eq:ineqform} implies upper bounds for both $m$ and $\sigmaplus $; cf. also the example \eqref{exofaintro} below.

Our second main result is about the asymptotic and the uniqueness (up to a shift) of the profile for a traveling wave with a given speed $c\geq c_*(\xi)$, $c\neq0$. 

\begin{theorem}\label{thm:asympandunic} 
Let $\xi\in S^{d-1} $ be fixed, and let conditions \eqref{as:chiplus_gr_m}--\eqref{firstmoment} hold. Let $c_*=c_*(\xi)\in\R$ be the minimal traveling wave speed given by \eqref{minspeed}, and let, for any speed $c\geq c_*$,  $\psi_c\in\M$ be the corresponding profile with the abscissa $\sigma(\psi_c)$.  
If \eqref{eq:critcase} holds and if, cf.~\eqref{eq:ineqform}, for $\sigmaplus=\sigma(\newaplus )$,
\begin{equation}\label{eq:critcasem}
    m=\kap \int_\R (1-\sigmaplus  s)\newaplus (s)e^{\sigmaplus  s}\,ds,
\end{equation}
we assume, additionally, that
\begin{equation}\label{secmomaddintro}
  \int_\R s^2 \newaplus (s)e^{\sigmaplus  s}\,ds<\infty.
\end{equation}
Let $c \geq c_*$ and $c\neq0$; then the following holds.
\begin{enumerate}[label=\arabic*)]
  \item There exists $D>0$, such that
\begin{equation}\label{eq:trw_asymptintro}
\psi_c(s)\sim D\,s^{j-1}\, e^{-\sigma(\psi_c) s},\quad s\to  \infty.
\end{equation}
Here $j=1$ in two cases: 1)~$c>c_*$; 2)~$c=c_*$ and \eqref{eq:critcase} holds as well as  the strict inequality in \eqref{eq:ineqform}. Otherwise, $j=2$, i.e. when $c=c_*$ and either \eqref{eq:critcase} fails or both \eqref{eq:critcase} and \eqref{eq:critcasem} hold. Moreover, $D=D_j$ may be chosen equal to $1$ by a shift of $\psi_c$.
\item If, additionally, there exist $\rho,\delta>0$, such that
\begin{assum}\label{as:aplus-aminus-is-pos}
    J_\theta(s)\geq\rho, \text{ for a.a. }  |s|\leq \delta,
\end{assum} 
then the traveling wave profile $\psi_c$ is unique up to a shift.
\end{enumerate}
\end{theorem} 

Clearly, \eqref{as:aplus-aminus-is-pos} implies that \eqref{as:a+nodeg-mod} holds with $r=0$.

Therefore, in the non-critical case, the profile of a traveling wave with a non-minimal speed decays exponentially at infinity with the rate equal to the abscissa of the profile, whereas for the minimal speed it decays slower: with an additional linear factor. However, in the critical case, the profile of the traveling wave with the minimal speed will not have that additional factor, unless both \eqref{eq:critcase} and \eqref{eq:critcasem} hold (and we can prove the latter under the additional assumption \eqref{secmomaddintro} only).

To demonstrate the critical case, consider the kernel 
\begin{equation}\label{exofaintro}
\newaplus (s):=\frac{\alpha e^{-\mu|s|}}{1+|s|^q}, \qquad s\in\R, \ q\geq0, \ \mu>0,
\end{equation}
where $\alpha>0$ is a normalizing constant. Then $\sigmaplus=\sigma(\newaplus )=\mu$.  
In Example~\ref{ex:spefunc} below, we will show that, for $q>2$, there exist $\mu_*>0$ and $m_*\in(0,\kap )$, such that
 $ \sigma(\psi_{c_*(\xi)})=\sigmaplus$, if only $\mu\in(0,\mu_*]$ and $m\in(0,m_*]$. 
The condition \eqref{secmomaddintro} does not take place only for $q\in(2,3]$, $\mu\in(0,\mu_*]$ and $m=m_*$.  

Another specific of the critical case is visible from the behavior on the positive half-line of the so-called characteristic function $\h_{\xi,c}$, corresponding to the traveling wave with a speed $c\geq c_*$, see \eqref{charfunc} and Proposition~\ref{prop:charfuncpropert} below: 
\[
\h_{\xi,c}(\la):= \kap (\L \newaplus )(\lambda)-m-c \lambda ,
\]
cf. e.g. \cite{Sch1980}. (This function is equal to infinity for $\lambda>\sigmaplus$.) Then the minimal positive root of $\h_{\xi,c}$ is $\sigma(\psi_c)$. The sketches on Figure~\ref{fig:sketches} reflect the difference between the critical and non-critical cases for the function $\h_{\xi,c}$.

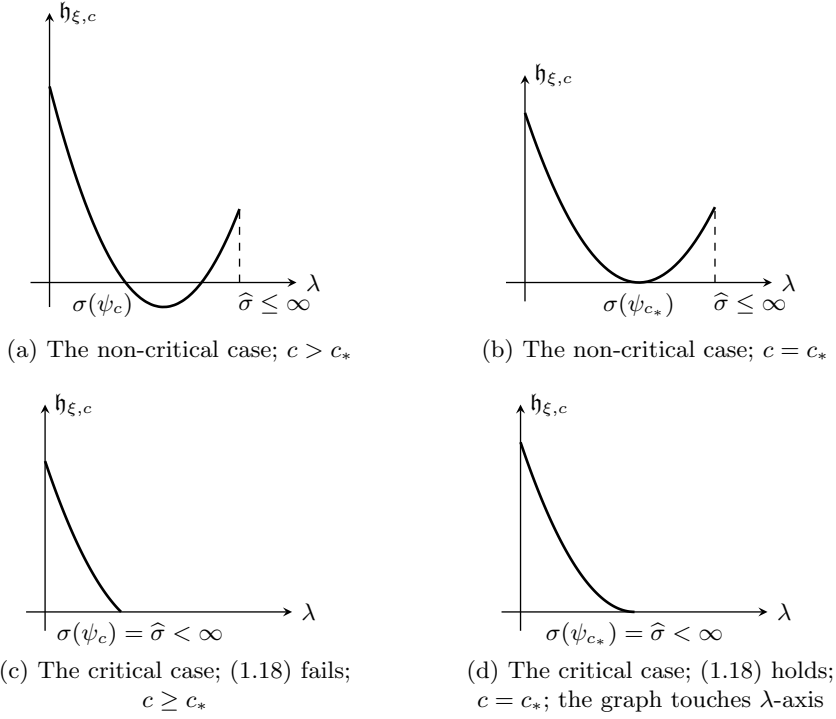
\begin{figure}[!ht]
\vspace{-0.75\baselineskip}
\centering
\begin{subfigure}[t]{.4\linewidth}
  \centering 
  \begin{tikzpicture}[xscale=0.5,yscale=0.65,line width=1pt,every node/.style={font=\small}, 
  declare function={u(\x)=0.5*\x*\x-3*\x+4;},>=stealth]
  \draw[line width=0.5pt,->] (-0.5,0)--(6.5,0) node[right] {$\lambda$};
  \draw[line width=0.5pt,->] (0,-0.5)--(0,5.5) node[right] {$\h_{\xi,c}$};
  \draw[domain=0:5,smooth,variable=\x,line width=1pt] plot (\x,{u(\x)});
  \node[below right] at (4.7,0) {$\sigmaplus\leq \infty$}; 
  \draw[dashed,line width=0.5pt] (5,0)--(5,{u(5)});
  \node[below] at (1.4,0) {$\sigma(\psi_c)$};
      \end{tikzpicture}
      \caption{\centering \small The non-critical case; $c>c_*$}
\end{subfigure}\qquad\qquad
\begin{subfigure}[t]{.4\linewidth}
\centering  \begin{tikzpicture}[xscale=0.5,yscale=0.5,line width=1pt,every node/.style={font=\small}, 
  declare function={u(\x)=0.5*\x*\x-3*\x+4.5;},>=stealth]
  \draw[line width=0.5pt,->] (-0.5,0)--(6.5,0) node[right] {$\lambda$};
  \draw[line width=0.5pt,->] (0,-0.5)--(0,5.5) node[right] {$\h_{\xi,c}$};
  \draw[domain=0:5,smooth,variable=\x,line width=1pt] plot (\x,{u(\x)});
  \node[below right] at (4.7,0) {$\sigmaplus\leq \infty$}; 
  \draw[dashed,line width=0.5pt] (5,0)--(5,{u(5)});
  \node[below] at (3,0) {$\sigma(\psi_{c_*})$};
      \end{tikzpicture}
      \caption{\centering \small The non-critical case; $c=c_*$}
\end{subfigure}


\begin{subfigure}[t]{.4\linewidth}
  \centering 
  \begin{tikzpicture}[xscale=0.5,yscale=0.5,line width=1pt,every node/.style={font=\small}, 
  declare function={u(\x)=0.5*\x*\x-3*\x+4;},>=stealth]
  \draw[line width=0.5pt,->] (-0.5,0)--(6.5,0) node[right] {$\lambda$};
  \draw[line width=0.5pt,->] (0,-0.5)--(0,5.5) node[right] {$\h_{\xi,c}$};
  \draw[domain=0:2,smooth,variable=\x,line width=1pt] plot (\x,{u(\x)});
  \node[below] at (2.5,0) {$\sigma(\psi_c)=\sigmaplus<\infty$};
      \end{tikzpicture}
      \caption{\centering \small The critical case; \eqref{eq:critcasem} fails; $c\geq c_*$}
\end{subfigure}\qquad\qquad
\begin{subfigure}[t]{.4\linewidth}
\centering  \begin{tikzpicture}[xscale=0.5,yscale=0.5,line width=1pt,every node/.style={font=\small}, 
  declare function={u(\x)=0.5*\x*\x-3*\x+4.5;},>=stealth]
  \draw[line width=0.5pt,->] (-0.5,0)--(6.5,0) node[right] {$\lambda$};
  \draw[line width=0.5pt,->] (0,-0.5)--(0,5.5) node[right] {$\h_{\xi,c}$};
  \draw[domain=0:3,smooth,variable=\x,line width=1pt] plot (\x,{u(\x)});
  \node[below] at (3,0) {$\sigma(\psi_{c_*})=\sigmaplus<\infty$};
      \end{tikzpicture}
      \caption{\centering \small The critical case;  \eqref{eq:critcasem} holds; $c=c_*$; the graph touches $\la$-axis}
\end{subfigure}
\vspace{-0.75\baselineskip}
\caption{\small Sketches of the characteristic function $\h_{\xi,c}$ for the critical (where \eqref{eq:critcase}--\eqref{eq:ineqform} hold) and the non-critical cases}\label{fig:sketches}
\end{figure}

In the case of the local nonlinearity in \eqref{eq:basic}, when $\kn=0$, the results of Theorems~\ref{thm:trwcandpsi}--\ref{thm:asympandunic} were mainly known in the literature under additional assumptions. For example, in~\cite{SZ2010}, the kernel $\aplus$  was synmmetric and compactly supported; in~\cite{CDM2008}, the kernel $\aplus$ was anisotropic, but $\newaplus $ was supposed to be compactly supported; whereas the conditions in \cite{ZLW2012} corresponded to a symmetric $\newaplus $, such that the inequality in \eqref{aplusexpint1} holds \emph{for all} $\mu>0$. 
In these both cases, $\sigmaplus =\sigma(\newaplus ) =\infty$; and hence, recall, $ \sigma(\psi_{c_*(\xi)})<\sigmaplus $. In~\cite{AGT2012}, an anisotropic kernel which satisfies \eqref{aplusexpint1} was considered (that allows $\sigmaplus <\infty$ as well), however, it was assumed that $ \sigma(\psi_{c_*(\xi)})<\sigmaplus $. The critical case $ \sigma(\psi_{c_*(\xi)})=\sigmaplus $, therefore, remained an open problem.

For a nonlocal nonlinearity in \eqref{eq:basic}, i.e. when $\kn\neq0$, the only known results \cite{YY2013} also concerned the more simple case $\sigma(\newaplus )=\infty$. 

The paper is organized as follows: in Section~\ref{sec:candpsi} we prove Theorem~\ref{thm:trwcandpsi} for both critical and non-critical cases, and in Section~\ref{sec:asympt_uniq} we discuss properties of the function $\h_{\xi,c}$ and prove Theorem~\ref{thm:asympandunic}.

\section{Speed and profile of a traveling wave}\label{sec:candpsi}

\subsection{Properties of the bilateral-type Laplace transform}

For an $f\in L^\infty(\R)$,
let $\L f$ be the bilateral-type Laplace transform of $f$ given by \eqref{Laplace},
cf.~\cite[Chapter~VI]{Wid1941}. We collect several results about $\L$ in the following lemma.
\begin{lemma}\label{lem:allaboutLapl}
Let $f\in L^\infty(\R)$.
\begin{enumerate}[label=\textnormal{(L\arabic*)}]
\item\label{L-converges} There exists $\sigma(f)\in [0,\infty]$ such that the integral \eqref{Laplace} converges in the strip $\{0<\Re z<\sigma(f)\}$ (provided that $\sigma(f)>0$) and diverges in the half plane $\{\Re z> \sigma(f)\}$ (provided that $\sigma(f)<\infty$).
\item\label{L-analytic} Let $\sigma(f)>0$. Then $(\L f) (z)$ is analytic in $\{0<\Re z<\sigma(f)\}$, and, for any $n\in\N$,
    \begin{equation*}
      \dfrac{d^n}{dz^n}(\L f)(z)=\int_\R e^{zs} s^n f(s)\,ds, \quad 0<\Re z<\sigma(f).
    \end{equation*}
\item\label{L-singular} Let $f\geq0$ a.e.\!\! and $0<\sigma(f)<\infty$. Then $(\L f)(z)$ has a singularity at $z=\sigma(f)$. In particular, $\L f$ has not an analytic extension to a strip  $0<\Re z<\nu$, with $\nu>\sigma(f)$.
\item\label{L-derivative} Let $f':=\frac{d}{ds}f\in L^\infty(\R)$, $f(\infty)=0$, and $\sigma(f')>0$. Then $\sigma(f)\geq\sigma(f')$ and, for any $0<\Re z<\sigma(f')$,
\begin{equation}\label{LaplaceofDer}
(\L f')(z)=-z (\L f)(z).
\end{equation}
\item\label{L-convolution} Let $g\in L^\infty(\R)\cap L^1(\R)$ and $\sigma(f)>0$, $\sigma(g)>0$. Then $\sigma(f*g)\geq\min\{\sigma(f),\sigma(g)\}$ and, for any $0<\Re z<\min\{\sigma(f),\sigma(g)\}$,
\begin{equation}\label{LaplaceofConv}
	\bigl(\L(f*g)\bigr)(z)=(\L f)(z) (\L g)(z).
\end{equation}
\item\label{L-onesidelimit}
Let $0\leq f\in L^1(\R)\cap L^\infty(\R)$ and $\sigma(f)>0$. Then 
\[
  \lim\limits_{\la\to 0+} (\L f)(\la)=\int_\R f(s)\,ds.
\]
\item\label{L-onesidelimit2}
Let $f\geq0$, $\sigma(f)\in(0,\infty)$ and $A:=\int_\R f(s) e^{\sigma(f)s}\,ds<\infty$. Then 
\[
  \lim\limits_{\la\to \sigma(f)-} (\L f)(\la)=A.
\]
\item\label{L-decaying} Let $f\geq0$ be decreasing on $\R$, and let $\sigma(f)>0$. Then, for any $0<\la<\sigma(f)$,
\begin{equation}\label{expdecay}
f(s)\leq \frac{\la e^\la}{e^\la-1} (\L f)(\la) e^{-\la s}, \quad s\in\R.
\end{equation}
Moreover, 
\begin{equation}\label{absofsq2}
	\sigma\bigl(f^2\bigr)\geq 2\sigma(f), 
\end{equation}
and for any $0\leq g\in L^\infty(\R)\cap L^1(\R)$, $\sigma(g)>0$,
\begin{equation}\label{absofsq}
  \sigma\bigl(f(g *f)\bigr)\geq\sigma(f)+ \min\bigl\{\sigma(g),\sigma(f)\bigr\}.
\end{equation}
\end{enumerate}
\end{lemma}
\begin{proof}
We can rewrite $\L=\L^++\L^-$, where
\begin{equation*}
(\L^\pm f)(z)=\int_{\R_\pm} f(s)e^{z s}\,ds, \quad \Re z>0,
\end{equation*}
$\R_+=[0,\infty)$, $\R_-=(\infty,0]$.  Let $\mathcal{L}$ denote the classical (unilateral) Laplace transform:
\[
(\mathcal{L}f)(z)=\int_{\R_+}f(s)e^{-z s}\,ds,
\]
and $\mathfrak{s}(f)$ be its abscissa of convergence (see details, e.g. in \cite[Chapter~II]{Wid1941}).
Then, clearly,
$(\L^+ f)(z)=(\mathcal{L}f)(-z)$, $(\L ^- f)(z)=(\mathcal{L}f^-)(z)$, where $f^-(s)=f(-s)$, $s\in\R$. As a result, $\sigma(f)=-\mathfrak{s}(f)$.

It is easily seen that, for $f\in L^\infty(\R)$, $\mathfrak{s}(f^-)\leq0$, in particular,the function $(\L^- f)(z)$ is analytic on $\Re z>0$.

Therefore, the properties \ref{L-converges}--\ref{L-singular} are direct consequences of \cite[Theorems II.1, II.5a, II.5b]{Wid1941}, respectively. The property \ref{L-derivative} may be easily derived from \cite[Theorem II.2.3a, II.2.3b]{Wid1941}, taking into account that $f(\infty)=0$. The property~\ref{L-convolution} one gets by a straightforward computation, cf. \cite[Theorem VI.16a]{Wid1941}; note that $f*g\in L^\infty(\R)$.

Next, $\sigma(f)>0$ implies $\mathfrak{s}(f)<0$, therefore, $\L^+f$ can be analytically continued to $0$. If $\mathfrak{s}(f^-)<0$, then $\L^- f$ can be analytically continued to $0$ as well, and \ref{L-onesidelimit} will be evident. Otherwise, if $\mathfrak{s}(f^-)=0$ then \ref{L-onesidelimit} follows from \cite[Theorem~V.1]{Wid1941}. Similar arguments prove \ref{L-onesidelimit2}.

To prove \ref{L-decaying} for a decreasing nonnegative $f$, note that, for any $0<\la<\sigma(f)$,
\[
f(s) \int_{s-1}^{s}e^{\la \tau}\,d\tau\leq
\int_{s-1}^s f(\tau) e^{\la \tau}\,d\tau\leq (\L f)(\la), \quad s\in\R,
\]
that implies \eqref{expdecay}. Next, by \ref{L-convolution},
$\sigma(g*f)>0$, and conditions on $g$ yield that $g*f\geq0$ is decreasing as well. Therefore, by \eqref{expdecay}, for any $0<\la<\sigma(g *f)$,
\begin{align*}
  \bigl\lvert\bigl(\L(f(g *f))\bigr)(z)\bigr\rvert&\leq \int_\R f(s)(g *f)(s)
  e^{s\Re z}\,ds\\
  &\leq \frac{\la e^\la}{e^\la-1}\bigl(\L (g *f)\bigr)(\la)\int_\R f(s)e^{-s\la}
  e^{s\Re z}\,ds<\infty,
\end{align*}
provided that $\Re z<\sigma(f)+\la<\sigma(f)+\sigma(g*f)$. As a result,
$\sigma\bigl( f(g*f)\bigr)\geq\sigma(f)+\sigma(g*f)$ that, by \ref{L-convolution}, implies \eqref{absofsq}. Similarly one can prove \eqref{absofsq2}.
\end{proof}

\subsection{Proof of Theorem~\ref{thm:trwcandpsi}}

Through the rest of the paper we will always assume that \eqref{as:chiplus_gr_m} holds.
Note also, that \eqref{as:aplus_gr_aminus} and \eqref{boundedkernels} imply $\newaminus\in L^\infty(\R)$.


\begin{remark} \label{shiftoftrw}
By \cite[Remark 3.6]{FKT100-1},  if $\psi\in\M$, $c\in\R$ gets \eqref{eq:deftrw} then, for any $s\in\R$, $\psi(\cdot+s)$ is a traveling wave to \eqref{eq:basic} with the same $c$.
\end{remark}

Fix any $\xi\in S^{d-1}  $. For $\mu>0$, we denote, cf. \eqref{apm1dim}, 
\begin{equation}\label{eq:dedAxi}
      \A_\xi(\mu):=\int_\X \aplus (x) e^{\mu x\cdot \xi}\,dx= \int_\R \newaplus (s) e^{\mu s}\,ds\in(0,\infty].
\end{equation}
Under \eqref{as:aplus_gr_aminus}, \eqref{aplusexpint1} and \eqref{boundedkernels}, $\sigma( \newaplusminus )>0$ and
\[
  \A_\xi(\mu)=(\L \newaplus )(\mu)<\infty, \quad 0<\mu<\sigma( \newaplus ).
\]

Consider, the following complex-valued function, cf. \eqref{aplusexpint1},
\begin{equation}\label{dedGxi}
G_\xi(z):=\frac{\kap (\L\newaplus )(z)-m}{z}, \quad \Re z>0,
\end{equation}
which is well-defined on $0<\Re z<\sigma(\newaplus )$. We have proved in \cite[formula (3.18)]{FKT100-1} that
\begin{equation}\label{ineqwillbeeq}
c_*(\xi)\leq \inf\limits_{\la>0} G_\xi(\la),
\end{equation}
where $c_*(\xi)$ is the minimal speed of traveling waves, cf.~Theorem~\ref{thm:trwexists}. We will show below that in fact there exists equality in \eqref{ineqwillbeeq}.

We start with the following notations to simplify the further statements.
\begin{definition}\label{def:Uxi}
  Let $m>0$, $\kap>0$, $\kl,\kn\geq0$, $0\leq \newaminus \in L^1(\R)$  be fixed, and \eqref{as:chiplus_gr_m} and \eqref{eq:nondegencomp} hold. For an arbitrary $\xi\in S^{d-1}  $, denote by $\Uxi$ the subset of functions $0\leq \newaplus \in L^1(\R)$ such that \eqref{as:aplus_gr_aminus}--\eqref{firstmoment} hold.

For $\newaplus\in\Uxi$, denote also the interval $I_\xi \subset (0,\infty)$ by
\begin{equation}\label{eq:maxint}
    I_\xi :=\begin{cases}
  (0,\infty), &\text{if}  \ \
  \sigma(\newaplus )=\infty,\\[2mm]
  \bigl(0,\sigma(\newaplus )\bigr), & \text{if} \ \sigma(\newaplus )<\infty \ \ \text{and} \ \ (\L \newaplus )\bigl(\sigma(\newaplus )\bigr)=\infty \\[2mm]
  \bigl(0,\sigma(\newaplus )\bigr], & \text{if} \  \sigma(\newaplus )<\infty \ \ \text{and} \ \   (\L \newaplus )\bigl(\sigma(\newaplus )\bigr)<\infty.
\end{cases}
\end{equation}
\end{definition}

\begin{proposition}\label{prop:infGisreached}
Let $\xi\in S^{d-1}  $ be fixed and $\newaplus\in\Uxi$.
Then there exists a unique $\la_*=\la_*(\xi)\in I_\xi $ such that
\begin{equation}\label{arginf}
  \inf\limits_{\la>0} G_\xi(\la)=\min_{\la\in I_\xi }G_\xi(\la)=G_\xi(\la_*)>\kap \m_\xi.
\end{equation}
Moreover, $G_\xi$ is strictly decreasing on $(0,\la_*]$ and $G_\xi$ is strictly increasing on $I_\xi\setminus(0,\la_*]$ (the latter interval may be empty).
\end{proposition}
\begin{proof}
We continue to use the notation $\sigmaplus:=\sigma(\newaplus )\in (0,\infty]$. Denote also
  \begin{equation}\label{defF}
 F_\xi(\la):=\kap  \A_\xi(\la)-m=\la G_\xi(\la), \qquad \la\in I_\xi .
 \end{equation}
By \ref{L-analytic}, for any $\la\in(0,\sigmaplus )$,
\begin{equation}\label{secderA}
  \A_\xi''(\la)=\int_\R s^2 \newaplus (s) e^{\la s}\,ds>0,
\end{equation}
therefore, $\A_\xi'(\la)$ is increasing on $(0,\sigmaplus )$; in particular, by \eqref{firstmoment}, we have, for any $\la\in(0,\sigmaplus )$,
\begin{equation}\label{posderL}
\int_\R s \newaplus (s) e^{\la s}\,ds = \A_\xi'(\la)> \A_\xi'(0)=\int_\R  s \newaplus (s)\,ds=\m_\xi.
\end{equation}
Next, by \ref{L-onesidelimit}, $F_\xi(0+)=\kap -m>0$, hence,
\begin{equation}\label{valat0}
  G_\xi(0+)=\infty.
\end{equation}
Finally, for $\la\in (0,\sigmaplus )$, we have
  \begin{align}\label{Gder1}
    G_\xi'(\la)&=\la^{-2}\bigl(\la F_\xi'(\la)-F_\xi(\la)\bigr)=\la^{-1}\bigl(F_\xi'(\la)-G_\xi(\la)\bigr),\\
    G_\xi''(\la)&=\la^{-1} (F_\xi''(\la)- 2G_\xi'(\la)).\label{Gder2}
  \end{align}

We will distinguish two cases.

\textit{Case 1.\/} There exists $\mu\in(0,\sigmaplus )$ with $G_\xi'(\mu)=0$. Then, by \eqref{Gder2}, \eqref{secderA},
\begin{equation*}
   G_\xi''(\mu)=\mu^{-1} F_\xi''(\mu)=\mu^{-1}\kap  \A_\xi''(\mu)>0.
\end{equation*}
Hence any stationary point of $G_\xi$ is with necessity a point of local minimum, therefore, $G_\xi$ has at most one such a point, thus it will be a global minimum. Moreover, by \eqref{Gder1}, \eqref{posderL}, $G'(\mu)=0$ implies
\begin{equation}\label{gatmin}
    G_\xi(\mu)=F_\xi'(\mu) =\kap  \A_\xi'(\mu)> \kap \m_\xi.
\end{equation}
Therefore, in the Case 1, one can choose $\la_*=\mu$ (which is unique then) to fulfill the statement.

List the conditions under which the Case 1 is possible.
 \begin{enumerate}
 \item Let $\sigmaplus =\infty$. Then, by \eqref{as:a+nodeg-mod},
\begin{equation}\label{toinfasltoinf}
  \frac{1}{\la}\A_\xi(\la)\geq \frac{1}{\la}\int_r^{r+\delta}\newaplus (s)e^{\la s}\,ds\geq\rho\frac{1}{\la^2}\bigl(e^{\la(r+\delta)}-e^{\la r}\bigr)\to\infty,
\end{equation}
as $\la\to\infty$. Then, in such a case, $G_\xi(\infty)=\infty$. Therefore, by \eqref{valat0}, there exists a zero of $G_\xi'$.

\item Let $\sigmaplus <\infty$ and $\A_\xi(\sigmaplus )=\infty$. Then, again, \eqref{valat0} implies the existence of a zero of $G_\xi'$ on $(0,\sigmaplus )$.

\item Let $\sigmaplus <\infty$ and $\A_\xi(\sigmaplus )<\infty$. By \eqref{defF}, \eqref{Gder1},
\begin{equation*}
\lim_{\la\to0+}\la^2 G_\xi'(\la)=-F_\xi(0+)=-(\kap -m)<0.
\end{equation*}
Therefore, the function $G_\xi'$ has a zero on $(0,\sigmaplus )$ if and only if takes a positive value at some point from $(0,\sigmaplus )$.
\end{enumerate}

Now, one can formulate and consider the opposite to the Case 1.

\textit{Case 2.\/} Let $\sigmaplus <\infty$, $\A_\xi(\sigmaplus )<\infty$, and
\begin{equation}\label{Gneg}
  G_\xi'(\la)<0,\quad \la\in(0,\sigmaplus ).
\end{equation}
Therefore,
\begin{equation}\label{infattheend}
\inf\limits_{\la>0} G_\xi(\la)=\inf\limits_{\la\in(0,\sigmaplus ]} G_\xi(\la)=\lim_{\la\to\sigmaplus -} G_\xi(\la)=G_\xi(\sigmaplus ),
\end{equation}
by \ref{L-onesidelimit2}. Hence we have the first equality in \eqref{arginf}, by setting $\la_*:=\sigmaplus $. To~prove the second inequality in \eqref{arginf}, note that, by \eqref{Gder1}, the inequality \eqref{Gneg} is equivalent to $F_\xi'(\la)< G_\xi(\la)$, $\la\in(0,\sigmaplus )$. Therefore, by \eqref{infattheend}, \eqref{defF}, \eqref{posderL},
  \[
  G_\xi(\sigmaplus )=\inf\limits_{\la\in(\frac{\sigmaplus }{2},\sigmaplus )} G_\xi(\la)\geq \inf\limits_{\la\in(\frac{\sigmaplus }{2},\sigmaplus )} F_\xi'(\la)\geq \kap  \A_\xi'\Bigl(\frac{\sigmaplus }{2}\Bigr)> \kap \m_\xi,
  \]
where we used again that, by \eqref{secderA}, $\A_\xi'$ and hence $F_\xi'$ are increasing on $(0,\sigmaplus )$.
The statement is fully proved now.
\end{proof}

The second case in the proof of Proposition~\ref{prop:infGisreached} requires additional analysis. Let $\xi\in S^{d-1}  $ be fixed and $\newaplus\in\Uxi$, $\sigmaplus :=\sigma(\newaplus )$. By \ref{L-analytic}, one can define the following function
\begin{equation}\label{specfunc}
    \T_\xi(\la):=\kap \int_\R (1-\la s)\newaplus (s)e^{\la s}\,ds\in\R, \qquad \la\in[0,\sigmaplus ).
\end{equation}
Note that
\begin{equation}\label{triv}
    \int_{\R_-}|s|\newaplus (s)e^{\sigmaplus  s}\,ds<\infty,
\end{equation}
and
$\int_{\R_+}s\newaplus (s)e^{\sigmaplus  s}\,ds\in(0,\infty]$ is well-defined. Then, in the case $\sigmaplus <\infty$ and $\A_\xi(\sigmaplus )<\infty$, one can continue $\T_\xi$ at $\sigmaplus $, namely,
  \begin{equation}\label{newquant}
    \T_\xi(\sigmaplus ):=\kap \int_\R (1-\sigmaplus  s)\newaplus (s)e^{\sigmaplus  s}\,ds\in[-\infty,\kap ).
  \end{equation}

To prove the latter inclusion, i.e. the strict inequality $\T_\xi(\sigmaplus )<\kap $, consider the function
$f_0(s):=(1-\sigmaplus  s)e^{\sigmaplus  s}$, $s\in\R$. Then, $f'_0(s)=-\sigmaplus^2 se^{\sigmaplus  s}$, and thus $f_0(s)<f_0(0)=1$, $s\neq 0$. Moreover, the function $g_0(s)=f_0(-s)-f_0(s)$, $s\geq0$ is such that $g_0'(s)=\sigmaplus ^2 s(e^{\sigmaplus  s}-e^{-\sigmaplus  s})>0$, $s>0$. As a result, for any $\delta>0$, $f_0(-\delta)>f_0(\delta)$, and
\[
\int_\R f_0(s)\newaplus (s)\,ds\leq f_0(-\delta)\int_{\R\setminus[-\delta,\delta]}\newaplus (s)\,ds
+\int_{[-\delta,\delta]}\newaplus (s)\,ds<\int_\R \newaplus (s)\,ds=1.
\]

\begin{proposition}\label{prop:Gnegholds}
  Let $\xi\in S^{d-1}  $ be fixed and $\newaplus\in\Uxi$. Suppose also that $\sigmaplus :=\sigma(\newaplus )<\infty$ and $\A_\xi(\sigmaplus )<\infty$. Then \eqref{Gneg} holds iff
   \begin{gather}\label{posquant}
    \T_\xi(\sigmaplus )\in(0,\kap ),\\
    \label{smallm}
     m\leq\T_\xi(\sigmaplus ).
   \end{gather}
\end{proposition}
\begin{proof}
  Define the function, cf. \eqref{defF},
\begin{equation}\label{defH}
  H_\xi(\la):=\la F'_\xi(\la)-F_\xi(\la), \quad \la\in(0,\sigmaplus ).
\end{equation}
By \eqref{Gder1}, the condition \eqref{Gneg} holds iff $H_\xi$ is negative on $(0,\sigmaplus )$.
By \eqref{defH}, \eqref{secderA}, one has $H'_\xi(\la)=\la F''_\xi(\la)>0$, $\la\in(0,\sigmaplus )$ and, therefore, $H_\xi$ is (strictly) increasing on $(0,\sigmaplus )$.
By~Proposition~\ref{prop:infGisreached}, $G_\xi'$, and hence $H_\xi$, are negative on a right-neighborhood of $0$. As a result, $H_\xi(\la)<0$ on $(0,\sigmaplus )$ iff
\begin{equation}\label{leftlimHisneg}
\lim\limits_{\la\to\sigmaplus -}H_\xi(\la)\leq0.
\end{equation}
On the other hand, by \eqref{defF}, \eqref{specfunc}, one can rewrite $H_\xi(\la)$ as follows:
\begin{equation}\label{rewrH}
  H_\xi(\la)=-\T_\xi(\la)+m, \quad \la\in(0,\sigmaplus ).
\end{equation}
By the monotone convergence theorem,
\[
\lim\limits_{\la\to\sigmaplus -}\int_{\R_+} (\la s-1)\newaplus (s)e^{\la s}\,ds=\int_{\R_+} (\sigmaplus  s-1)\newaplus (s)e^{\sigmaplus  s}\,ds\in(-1,\infty].
\]
Therefore, by \eqref{triv}, \eqref{rewrH}, $\T_\xi(\sigmaplus )\in\R$ iff $H_\xi(\sigmaplus )=\lim\limits_{\la\to\sigmaplus -}H_\xi(\la)\in\R$. Next, clearly, $H_\xi(\sigmaplus )\in(m-\kap ,0]$ holds true iff both \eqref{smallm} and \eqref{posquant} hold.

As a result, \eqref{Gneg} is equivalent to \eqref{leftlimHisneg} and the latter, by \eqref{triv}, implies that $\T_\xi(\sigmaplus )\in\R$ and hence $H_\xi(\sigmaplus )\in(m-\kap ,0]$.
Vice versa, \eqref{posquant} yields $\T_\xi(\sigmaplus )\in\R$ that together with \eqref{smallm} give that $H_\xi(\sigmaplus )\leq0$, i.e. that \eqref{Gneg} holds.
\end{proof}

According to the above, it is natural to consider two subclasses of functions from $\Uxi$, cf.~Definition~\ref{def:Uxi}.

\begin{definition}\label{def:VWxi}
  Let $\xi\in S^{d-1}  $ be fixed. We denote by $\Vxi$ the class of all kernels $\newaplus\in\Uxi$ such that one of the following assumptions does hold:
  \begin{enumerate}
    \item $\sigmaplus :=\sigma(\newaplus )=\infty$;
    \item $\sigmaplus <\infty$ and $\A_\xi(\sigmaplus )=\infty$;
    \item $\sigmaplus <\infty$, $\A_\xi(\sigmaplus )<\infty$ and
  $\T_\xi(\sigmaplus )\in[-\infty,m)$, where $\T_\xi(\sigmaplus )$ is given by \eqref{newquant}.
  \end{enumerate}
\noindent Correspondingly, we denote by $\Wxi$ the class of all kernels $\newaplus\in\Uxi$ such that
  $\sigmaplus <\infty$, $\A_\xi(\sigmaplus )<\infty$, and $\T_\xi(\sigmaplus )\in[m,\kap )$.
  Clearly, $\Uxi=\Vxi\cup\Wxi$.
\end{definition}

As a result, combining the proofs and statements of Propositions~\ref{prop:infGisreached} and~\ref{prop:Gnegholds}, one immediately gets the following corollary.
\begin{corollary}\label{cor_alltogether}
   Let $\xi\in S^{d-1}  $ be fixed, $\newaplus\in\Uxi$, and $\la_*$ be the same as in Proposition~\ref{prop:infGisreached}. Then $\la_*<\sigmaplus :=\sigma(\newaplus )$ iff $\newaplus\in\Vxi$; moreover, then $G'(\la_*)=0$. Correspondingly, $\la_*=\sigmaplus $ iff $\newaplus\in\Wxi$; in this case,
   \begin{equation}\label{leftlimG}
     \lim\limits_{\la\to\sigmaplus -}G_\xi'(\la)=\frac{m-\T_\xi(\sigmaplus )}{(\sigmaplus )^2}\leq 0.
   \end{equation}
\end{corollary}

\begin{example}\label{ex:spefunc}
To demonstrate the cases of Definition~\ref{def:VWxi} on an example, consider the following family of functions, cf. \eqref{exofaintro},
\begin{equation}\label{exofa}
\newaplus (s):=\frac{\alpha e^{-\mu|s|^p}}{1+|s|^q}, \quad s\in\R, p\geq0, q\geq0, \mu>0,
\end{equation}
where $\alpha>0$ is a normalizing constant to get $\int_\R \newaplus (s)\,ds=1$. Clearly, the case $p\in[0,1)$ implies $\sigma(\newaplus )=0$, that is impossible under assumption \eqref{aplusexpint1}. Next, $p>1$ leads to $\sigma(\newaplus )=\infty$, in particular, the corresponding $\newaplus\in\Vxi$. Let now $p=1$, then $\sigma(\newaplus )=\mu$. The case $q\in[0,1]$ gives $\A_\xi(\sigmaplus )=\infty$, i.e. $\newaplus\in\Vxi$ as well. In~the~case $q\in (1,2]$, we will have that $\A_\xi(\sigmaplus )<\infty$, however,$\int_\R s \newaplus (s)e^{\mu s}\,ds=\infty$, i.e. $\T_\xi(\mu)=-\infty$, and again $\newaplus\in\Vxi$. Let $q>2$; then, by \eqref{specfunc},
\begin{align*}
\T_\xi(\mu)&=\kap \alpha \int_{\R_-} \frac{1-\mu s}{1+|s|^q}e^{2\mu s}\,ds
+\kap \alpha \int_{\R_+} \frac{1-\mu s}{1+s^q}\,ds\\&\geq \kap \alpha \int_{\R_+} \frac{1-\mu s}{1+s^q}\,ds=\frac{\pi\kap \alpha }{q}\biggl(\frac{1}{\sin\frac{\pi}{q}}-\frac{\mu}{\sin\frac{2\pi}{q}}\biggr)\geq m,
\end{align*}
if only $\mu\leq 2\cos\frac{\pi}{q}- \frac{m q}{\kap \alpha\pi}\sin\frac{2\pi}{q}$ (note that $q>2$ implies $\sin\frac{2\pi}{q}>0$); then we have $\newaplus\in\Wxi$. On the other hand, using the inequality $te^{-t}\leq e^{-1}$, $t\geq0$, one gets
\begin{align}\label{muexpr}
  \T_\xi(\mu)&=\kap \alpha \int_{\R_+} \frac{(1+\mu s)e^{-2\mu s}+1-\mu s}{1+s^q}\,ds\\
  &\leq \kap \alpha \int_{\R_+} \frac{1+\frac{1}{2e}+1-\mu s}{1+s^q}\,ds=
  \frac{\pi\kap \alpha }{q}\biggl(\frac{1+4e}{2e\sin\frac{\pi}{q}}-\frac{\mu}{\sin\frac{2\pi}{q}}\biggr)<m,  \notag
\end{align}
if only $\mu>\frac{1+4e}{e}\cos\frac{\pi}{q}- \frac{m q}{\kap \alpha\pi}\sin\frac{2\pi}{q}$; then we have $\newaplus\in\Vxi$.
 Since
 \[
 \frac{d}{d\mu}\bigl((1+\mu s)e^{-2\mu s}+1-\mu s\bigr)=-se^{-2\mu s}(1+2s\mu)-s<0, \quad s>0, \mu>0,
 \]
we have from \eqref{muexpr}, that $\T_\xi(\mu)$ is strictly decreasing and continuous in $\mu$, therefore, there exist a critical value
\[
\mu_*\in \Bigl(2\cos\frac{\pi}{q}- \frac{m q}{\kap \alpha\pi}\sin\frac{2\pi}{q},(4+e^{-1})\cos\frac{\pi}{q}- \frac{m q}{\kap \alpha\pi}\sin\frac{2\pi}{q}\Bigr),
\]
such that, for all $\mu>\mu_*$, $\newaplus\in\Vxi$, whereas, for $\mu\in(0,\mu_*]$, $\newaplus\in\Wxi$.
\end{example}

Now we are ready to prove the main statement of this subsection.

\begin{theorem}\label{thm:speedandprofile}
 Let $\xi\in S^{d-1}  $ be fixed and $\newaplus\in\Uxi$. Let $c_*(\xi)$ be the minimal traveling wave speed according to Theorem~\ref{thm:trwexists}, and let, for any $c\geq c_*(\xi)$, the function $\psi=\psi_c\in\M$ be a traveling wave profile corresponding to the speed~$c$. Let $\la_*\in I_\xi$ be the same as in Proposition~\ref{prop:infGisreached}. Denote, as usual, $\sigmaplus :=\sigma(\newaplus )$. Then
\begin{enumerate}
\item Theorem \ref{thm:trwcandpsi} holds.
\item For $\newaplus\in\Vxi$, one has $\la_*<\sigmaplus $ and there exists another representation for the minimal speed than \eqref{speedviaabs}, namely,
\begin{equation}\label{minspeed-spec}
    \begin{split}c_*(\xi)&= \kap \int_{\X} x\cdot \xi \, \aplus (x) e^{\la_* x\cdot \xi} \,dx\\&=\kap \int_\R s\newaplus (s)e^{\la_* s}\,ds>\kap \m_\xi.
    \end{split}
\end{equation}
Moreover, for all $\la\in(0,\la_*]$,
\begin{equation}\label{estonm}
\T_\xi(\la)\geq m,
\end{equation}
and the equality holds for $\la=\la_*$ only.
\item For $\newaplus\in\Wxi$, one has $\la_*=\sigmaplus $.
Moreover, the inequality \eqref{estonm} also holds as well as, for all $\la\in(0,\la_*]$,
\begin{equation}
    c\geq \kap \int_\R s\newaplus (s)e^{\la s}\,ds,\label{estonc}
\end{equation}
whereas the equalities in \eqref{estonm} and \eqref{estonc} hold true now for $m=\T_\xi(\sigmaplus )$, $\la=\la_*$, $c=c_*(\xi)$ only.
\end{enumerate}
\end{theorem}
\begin{proof}
By Theorem~\ref{thm:trwexists}, for any $c\geq c_*(\xi)$,
there exists a profile $\psi\in\M$, cf.~Remark~\ref{shiftoftrw}, which define a traveling wave solution \eqref{eq:deftrw} to \eqref{eq:basic} in the direction $\xi$. Then, by \eqref{eq:trw}, we get
\begin{multline}
	-c\psi'(s) = \kap (\newaplus*\psi)(s) - m\psi(s) \\- \kl\psi^2(s)
			-\kn\psi(s) (\newaminus*\psi)(s), \ s\in\R.\label{eq:trwder}
\end{multline}

\textit{Step 1}.
By \eqref{eq:trwexpint}, we have that $\sigma(\psi)>0$.
Rewrite \eqref{as:aplus_gr_aminus} as follows
\begin{equation}\label{acheckpos}
  \kap \newaplus (s)\geq \kn \theta \newaminus (s), \quad \text{a.a.} \ s\in\R,
\end{equation}
therefore, $\sigma(\newaminus )\geq \sigma(\newaplus )>0$, if $\kn>0$. Take any $z\in\mathbb{C}$ with
\begin{multline}\label{unexpineq}
	\quad 0<\Re z<\min\bigl\{\sigma(\newaplus ),\sigma(\psi)\bigr\}\leq \sigma(\psi)
		\\<\min\{ \sigma(\psi^2),\sigma\bigl(\psi(\newaminus *\psi)\bigr)\},\quad 
\end{multline}
where the later inequality holds by \eqref{absofsq2} and \eqref{absofsq}. As a result, by \ref{L-convolution}, \ref{L-decaying}, being multiplied on $e^{z s}$ the l.h.s.\,of \eqref{eq:trwder} will be  integrable (in $s$) over $\R$. Hence,~for~any~$z$ which satisfies \eqref{unexpineq}, $(\L \psi')(z)$ converges. By~\ref{L-derivative}, it yields $\sigma(\psi)\geq\sigma(\psi')\geq\min\bigl\{\sigma(\newaplus ),\sigma(\psi)\bigr\}$.

Therefore, by \eqref{LaplaceofDer}, \eqref{LaplaceofConv}, we get from \eqref{eq:trwder}
\begin{multline}
	c z  (\L \psi)(z)=\kap (\L \newaplus )(z) (\L\psi)(z)-m(\L\psi)(z)\\
		-\kl\bigl( \L(\psi^2) \bigr)(z) -\kn\bigl(\L(\psi(\newaminus *\psi))\bigr)(z),\label{eqforLtr}
\end{multline}
if only
\begin{equation}\label{condz}
0<\Re z<\min\bigl\{\sigma(\newaplus ),\sigma(\psi)\bigr\}.
\end{equation}

Since $\psi\not\equiv 0$, we have that $(\L \psi)(z)\neq 0$, therefore, one can rewrite \eqref{eqforLtr} as follows
\begin{equation}\label{neweq}
	G_\xi(z)-c = \frac{\kl \bigl( \L(\psi^2) \bigr)(z)+\kn\bigl(\L(\psi(\newaminus *\psi))\bigr)(z)}{z(\L \psi)(z)},
\end{equation}
if \eqref{condz} holds.
By~\eqref{unexpineq}, both nominator and denominator in the r.h.s. of \eqref{neweq} are analytic on $0<\Re z<\sigma(\psi)$, therefore.
Suppose that $\sigma(\psi)>\sigma(\newaplus )$, then \eqref{neweq} holds on $0<\Re z<\sigma(\newaplus )$, however, the r.h.s. of \eqref{neweq} would be analytic at $z=\sigma(\newaplus )$, whereas, by \ref{L-singular}, the l.h.s. of \eqref{neweq} has a singularity at this point. As a result,
\begin{equation}\label{absofkernelisbigger}
\sigma(\newaplus )\geq\sigma(\psi),
\end{equation}
for any traveling wave profile $\psi\in\M$.
Thus one gets that \eqref{neweq} holds true on
$0<\Re z<\sigma(\psi)$.

Prove that
\begin{equation}\label{finiteabswavenew}
\sigma(\psi)<\infty.
\end{equation}
	Since $0\leq\psi\leq\theta$ yields $0\leq \aminus *\psi\leq\theta$, one gets from \eqref{neweq} that, for any $0<\la<\sigma(\psi)$,
\begin{equation}
c\geq G_\xi(\la)-\kam\dfrac{\theta}{\la}
=\frac{\kap (\L\newaplus )(\la)-\kap }{\la}.\label{lowestforspeed}
\end{equation}
If $\sigma(\newaplus )<\infty$ then \eqref{finiteabswavenew} holds by \eqref{absofkernelisbigger}. Suppose that $\sigma(\newaplus )=\infty$. By \eqref{toinfasltoinf}, the r.h.s. of \eqref{lowestforspeed} tends to $\infty$ as $\la\to\infty$, thus the latter inequality cannot hold for all $\la>0$; and, as a result, \eqref{finiteabswavenew} does hold.

\textit{Step 2}. Recall that \eqref{ineqwillbeeq} holds. Suppose that $c\geq c_*(\xi)$ is such that, cf.~\eqref{arginf},
\begin{equation}\label{cbiginf}
c\geq G_\xi(\la_*)=\inf_{\sigmaplus \in (0,\la_*] } G_\xi(\la)=\inf_{\sigmaplus \in I_\xi } G_\xi(\la).
\end{equation}
Then, by Proposition~\ref{prop:infGisreached}, the equation $G_\xi(\la)=c$, $\la\in I_\xi $, has one or two solutions. Let $\la_c$ be the unique solution in the first case or the smaller of the solutions in the second one. Since $G_\xi$ is decreasing on $(0,\la_*]$, we have $\la_c\leq \la_*$. Since the nominator in the r.h.s. of \eqref{neweq} is positive, we immediately get from \eqref{neweq} that
\begin{equation}\label{eqdopinfty}
(\L\psi)(\la_c)=\infty,
\end{equation}
therefore, $\la_c\geq\sigma(\psi)$.
On the other hand, one can rewrite \eqref{neweq} as follows
\begin{equation}\label{rewriting}
	(\L\psi)(z) = \frac{\kl\bigl( \L(\psi^2) \bigr)(z)+\kn\bigl(\L(\psi(\newaminus *\psi))\bigr)(z)}{z(G_\xi(z)-c)}.
\end{equation}
By \eqref{neweq}, $G_\xi(z)\neq c$, for all $0<\Re z<\sigma(\psi)\leq\la_c\leq\la_*\leq\sigma(\newaplus )$. As a result, by \eqref{unexpineq}, \ref{L-converges}, and \ref{L-singular}, $\la_c=\sigma(\psi)$, that together with \eqref{eqdopinfty} proves \eqref{finitespeed} and \eqref{Laplinfatabs}, for waves whose speeds satisfy \eqref{cbiginf}. By \eqref{aplusexpint1}, \eqref{eq:dedAxi},  we immediately get, for such speeds, \eqref{speedviaabs} as well. Moreover, \eqref{speedviaabs} defines a strictly monotone function $(0,\la_*]\ni\sigma(\psi)\mapsto c\in[G_\xi(\la_*),\infty)$.

Next, by \eqref{specfunc}, \ref{L-analytic}, \eqref{defF}, \eqref{Gder1}, we have that, for any $\la\in I_\xi$,
\begin{equation}
  \T_\xi(\la)=\kap \A_\xi(\la)-\kap \la\A_\xi'(\la)=
m+F_\xi(\la)-\la F_\xi'(\la)=m-\la^2 G_\xi'(\la).\label{Txirepr}
\end{equation}
Recall that, by Proposition~\ref{prop:infGisreached}, the function $G_\xi$ is strictly decreasing on $(0,\la_*)$. Then \eqref{Txirepr} implies that $\T_\xi(\la)>m$, $\la\in(0,\la_*)$. On the other hand, by the second equality in \eqref{Gder1}, the inequality $G_\xi'(\la)<0$, $\la\in(0,\la_*)$, yields $G_\xi(\la)>F_\xi'(\la)$, for such a $\la$.
Let $c>G_\xi(\la_*)$. By \eqref{speedviaabs}, \eqref{defF}, we have then
$c>\kap \A_\xi'(\la)$, for all $\la\in[\sigma(\psi),\la_*)$. By~\eqref{secderA}, $F_\xi'$ is increasing, hence, by~\ref{L-analytic}, the strict inequality in \eqref{estonc} does hold, for $\la\in(0,\la_*)$.

 Let again $c\geq G_\xi(\la_*)$, and let $\newaplus\in\Vxi$. Then, by~Corollary~\ref{cor_alltogether},
 $\la_*<\sigma(\newaplus )$ and $G'(\la_*)=0$. By~\eqref{Gder1}, the latter equality and \eqref{Txirepr} give $\T_\xi(\la_*)=m$, that fulfills the proof of \eqref{estonm}, for such $\newaplus$ and $m$. Moreover, by \eqref{gatmin},
 \begin{equation}\label{dopequal}
   G_\xi(\la_*)=\kap \A_\xi'(\la_*)=\kap \int_\R s\newaplus (s)e^{\la_* s}\,ds.
 \end{equation}
 Let $\newaplus\in\Wxi$, then $\la_*=\sigma(\newaplus )$.
 It means that $\T_\xi(\la_*)=m$ if $m=\T_\xi(\sigmaplus )$ only, otherwise, $\T_\xi(\la_*)>m$.
 Next, we get from \eqref{cbiginf}, \eqref{Gder1}
 \eqref{leftlimG},
 \begin{equation}\label{dsgdsggdetwet}
  c\geq G_\xi(\la_*)\geq \lim\limits_{\la\to\la_*-}F_\xi'(\la)
  =\kap \int_\R s\newaplus (s)e^{\la_* s}\,ds,
 \end{equation}
 where the latter equality may be easily verified if we rewrite, for $\la\in(0,\la_*)$,
 \begin{equation}\label{tew463634436}
   F_\xi'(\la)=\kap \int_{\R_-}s\newaplus (s)e^{\la s}\,ds
 +\kap \int_{\R_+}s\newaplus (s)e^{\la s}\,ds,
 \end{equation}
 and apply the dominated convergence theorem to the first integral and the monotone convergence theorem for the second one. On the other hand, \eqref{leftlimG} implies that the second inequality in \eqref{dsgdsggdetwet} will be strict iff $m<\T_\xi(\sigmaplus )$, whereas, for $c=G_\xi(\la_*)=\inf\limits_{\la>0}G_\xi(\la)$ and $m=\T_\xi(\sigmaplus )$, we will get all equalities in \eqref{dsgdsggdetwet}.

\textit{Step 3}. Let now $c\geq c_*(\xi)$ and suppose that $\sigma(\newaplus )>\sigma(\psi)$.
Prove that \eqref{cbiginf} does hold.
On the contrary, suppose that the $c$ is such that
\begin{equation}\label{hypoteticspeed}
c_*(\xi)\leq c<\inf_{\la\in (0,\la_*] } G_\xi(\la)=\inf_{\la>0 } G_\xi(\la).
\end{equation}
Again, by \eqref{neweq}, $G_\xi(z)\neq c$, for all $0<\Re z<\sigma(\psi)$, and \eqref{rewriting} holds, for such a~$z$.
Since we supposed that $\sigma(\newaplus )>\sigma(\psi)$, one gets from
\eqref{unexpineq}, that both
nominator and denominator of the r.h.s. of \eqref{rewriting} are analytic on
\[
\{0<\Re z<\nu\} \supsetneq \{0<\Re z<\sigma(\psi)\},
\]
where $\nu=\min\bigl\{ \sigma(\newaplus ),\sigma\bigl(\psi(\newaminus *\psi)\bigr), \sigma(\psi^2) \bigr\}$.
On the other hand, \ref{L-singular} implies that $\L\psi$ has a singularity at $z=\sigma(\psi)$. Since 
\[
	\min \{ \bigl(\L(\psi^2)\bigr) (\sigma(\psi)) , \bigl(\L(\psi(\newaminus *\psi))\bigr)(\sigma(\psi)) \} > 0,
\]
the equality \eqref{rewriting} would be possible if only $G_\xi(\sigma(\psi))=c$, that contradicts \eqref{hypoteticspeed}.

\textit{Step 4}. By \eqref{absofkernelisbigger}, it remains to prove that, for $c\geq c_*(\xi)$, \eqref{cbiginf} does holds, provided that we have $\sigma(\newaplus )=\sigma(\psi)$.
Again on the contrary, suppose that \eqref{hypoteticspeed} holds.
For $0<\Re z<\sigma(\psi)$, we can rewrite \eqref{eqforLtr} as follows
\begin{equation}\label{rewriting2}
z(\L\psi)(z)(G_\xi(z)-c) = \kl\bigl( \L(\psi^2) \bigr)(z) + \kn\bigl(\L(\psi(\newaminus *\psi))\bigr)(z).
\end{equation}
In the notations of the proof of Lemma~\ref{lem:allaboutLapl}, the functions
$\L^-\psi$ and $\L^-\newaplus $ are analytic on $\Re z>0$. Moreover,
$(\L^+\psi)(\la)$ and $(\L^+\newaplus )(\la)$ are increasing on $0<\la<\sigma(\newaplus )=\sigma(\psi)$.
Then, cf.~\eqref{tew463634436}, by the monotone convergence theorem, we will get from \eqref{rewriting2} and \eqref{unexpineq}, that
\begin{equation}\label{finiteleftlims}
  \int_\R \psi(s) e^{\sigma(\psi)s}\,ds<\infty, \qquad
\int_\R \newaplus (s) e^{\sigma(\newaplus )s}\,ds<\infty.
\end{equation}

We are going to apply now \cite[Proposition 2.10]{FKT100-1}, in the case $d=1$, to the equation
\begin{equation}
  \begin{cases}
    \begin{aligned}
      \dfrac{\partial \phi}{\partial t}(s,t)&=\kap (\newaplus *\phi)(s,t)-m\phi(s,t) -\kl \phi^2(s,t) \\&\quad
        -\kn\phi(s,t)(\newaminus *\phi)(s,t), \qquad t>0, \ \mathrm{a.a.}\ s\in\R,
    \end{aligned}\\
    \phi(s,0)=\psi(s),\qquad \mathrm{a.a.}\ s\in\R,
  \end{cases}\label{eq:basic_one_dim}
\end{equation}
 where the initial condition $\psi$ is a wave profile with the speed $c$ which satisfies \eqref{hypoteticspeed}. Namely, we set $\Delta_R:=(-\infty,R)\nearrow\R$, $R\to\infty$ and
\begin{align}
  a_R^{\pm}(s)&:=\1_{\Delta_R}(s)a^{\pm}(s),\quad s\in\R,\label{trkern}\\
  A_R^\pm&:=\int_{\Delta_R}a^\pm(x)\,dx \nearrow 1, \quad R\to\infty.\label{ARdef}
\end{align}
Consider a strictly monotone sequence $\{R_n\mid n\in\N\}$, such that $0<R_n\to\infty$, $n\to\infty$ and
\begin{equation}\label{iiaslater}
  A_{R_n}^+>\frac{m}{\kap }\in(0,1).
\end{equation}
Let $\theta_n:=\theta_{R_n}$ be given by 
\begin{equation}\label{defofthetaR}
  \theta_{R_n}=\dfrac{\ka^+A_{R_n}^+-m}{\kn A_{R_n}^- + \kl}\to \theta, \quad R_n\to\infty.
\end{equation}
Then, by \cite[formula (2.17)]{FKT100-1}, $\theta_n\leq \theta$, $n\in\N$. 

Fix an arbitrary $n\in\N$. Consider the `truncated' equation \eqref{eq:basic_one_dim} with $\newaplusminus$ replaced by $\newaplusminus_{R_n}$, and the initial condition $w_0(s):=\min\{\psi(s),\theta_n\}\in C_{ub}(\R)$. By~\cite[Proposition~2.10]{FKT100-1}, there exists the unique solution $w^{(n)}(s,t)$ of the latter equation. Moreover, if we denote the corresponding nonlinear mapping by~$\tilde{Q}^{(n)}_t$, we will have from \cite[formulas (2.15)--(2.16)]{FKT100-1}, that
\begin{equation}\label{saewtwetwtgvxc}
  (\tilde{Q}^{(n)}_t w_0)(s)\leq\theta_n, \quad s\in\R, t\geq0,
\end{equation}
 and
\begin{equation}\label{s4366643}
  (\tilde{Q}^{(n)}_t w_0)(s)\leq \phi(s,t),
\end{equation}
where $\phi$ solves \eqref{eq:basic_one_dim}. 

By~\cite[Remark 3.4]{FKT100-1}, we get from \eqref{s4366643} that $(\tilde{Q}^{(n)}_1 w_0)(s+c)\leq \psi(s)$, $s\in\R$. The latter inequality together with \eqref{saewtwetwtgvxc} imply
\begin{equation}\label{wavedoesexists}
  (\tilde{Q}^{(n)}_1 w_0)(s+c)\leq w_0(s).
\end{equation}
Then, by the same arguments as in the proof of \cite[Theorem 1.1]{FKT100-1}, we obtain from \cite[Theorem 5]{Yag2009} that there exists a traveling wave $\psi_n$ for the equation \eqref{eq:basic_one_dim} with $\newaplusminus$ replaced by $\newaplusminus_{R_n}$, whose speed will be exactly $c$ (and $c$ satisfies \eqref{hypoteticspeed}).

Now we are going to get a contradiction, by proving that
\begin{equation}\label{infsareequal}
  \inf_{\la>0}G_\xi(\la)=\lim_{n\to\infty}\inf_{\la>0}G^{(n)}_\xi(\la),
\end{equation}
where $G^{(n)}_\xi$ is given by \eqref{dedGxi} with $ \newaplusminus $ replaced by
$a_{n}^\pm:=a_{R_n}^\pm$.
The sequence of functions $G_\xi^{(n)}$ is point-wise monotone in $n$ and it converges to $G_\xi$ point-wise, for $0<\la\leq \sigma(\newaplus )$; note we may include $\sigma(\newaplus )$ here, according to \eqref{finiteleftlims}. Moreover, $G_\xi^{(n)}(\la)\leq G_\xi(\la)$, $0<\la\leq \sigma(\newaplus )$. As a result, for any $n\in\N$,
\begin{equation}\label{asft346eyeryd}
G_\xi^{(n)}(\la_*^{(n)})=\inf_{\la>0}G_\xi^{(n)}(\la)
\leq \inf_{\la>0}G_\xi(\la)=G_\xi(\la_*).
\end{equation}
Hence if we suppose that \eqref{infsareequal} does not hold, then
\[
\inf\limits_{\la>0}G_\xi(\la)-\lim\limits_{n\to\infty}\inf\limits_{\la>0}G^{(n)}_\xi(\la)>0.
\]
Therefore, there exist $\delta>0$ and $N\in\N$, such that
\begin{equation}\label{wqwqwqwqwe}
  G_\xi^{(n)}(\la_*^{(n)})=\inf\limits_{\la>0}G^{(n)}_\xi(\la)\leq \inf\limits_{\la>0}G_\xi(\la)-\delta=G_\xi(\la_*)-\delta, \quad n\geq N.
\end{equation}

Clearly, \eqref{trkern} with $\Delta_{R_n}=(-\infty,R_n)$ implies that $\sigma(\newaplus _n)=\infty$, hence $G^{(n)}_\xi$ is analytic on~\mbox{$\Re z>0$}. One can repeat all considerations of the first three steps of this proof for the equation \eqref{eq:basic_one_dim}. Let $c^{(n)}_*(\xi)$ be the corresponding minimal traveling wave speed, according to Theorem~\ref{thm:trwexists}. Then the corresponding inequality \eqref{finiteabswavenew} will show that the abscissa of an arbitrary traveling wave to \eqref{eq:basic_one_dim} (with $\newaplusminus$ replaced by $\newaplusminus_{R_n}$) is less than $\sigma(a_{n}^+)=\infty$. As a result, the inequality
$c^{(n)}_*(\xi)<\inf\limits_{\la>0 } G^{(n)}_\xi(\la)$, cf.~\eqref{hypoteticspeed}, is impossible, and hence, by the Step~3,
\begin{equation}\label{minspeedfortrunc}
  c\geq c^{(n)}_*(\xi)=\inf_{\la>0 } G^{(n)}_\xi(\la)=G^{(n)}_\xi(\la_*^{(n)}),
\end{equation}
where $\la_*^{(n)}$ is the unique zero of the function $\frac{d}{d\la}G^{(n)}_\xi (\la)$. Let $\T_\xi^{(n)}$ be given on $(0,\infty)$ by \eqref{specfunc} with $\newaplus $ replaced by
${a}_n^+$. Then
\begin{equation}\label{derofT}
  \frac{d}{d\la}\T_\xi^{(n)}(\la)=-\la\kap  \int_{-\infty}^{R_n} \newaplus (s) s^2 e^{\la s}\,ds<0, \quad \la>0.
\end{equation}
By \eqref{estonm}, the unique point of intersection of the strictly decreasing function
$y=\T_\xi^{(n)}(\la)$ and the horizontal line $y=m$ is exactly the point  $(\la_*^{(n)},0)$.

Prove that there exist $\la_1>0$, such that $\la_*^{(n)}>\la_1$, $n\geq N$, and there exists $N_1\geq N$, such that $\T_\xi^{(n)}(\la)\leq \T_\xi^{(m)}(\la)$, $n> m\geq N_1$, $\la\geq\la_1$.
Recall that \eqref{iiaslater} holds; we have
\begin{align*}
  \la G_\xi^{(n)}(\la) & = \kap \int_\R \newaplus _n(s) (e^{\la s}-1)\,ds +\kap  {A}^+_{R_n}-m\\
  & \geq \kap \int_{-\infty}^0 \newaplus _n(s) (e^{\la s}-1)\,ds +\kap  A^+_{R_1}-m,
\end{align*}
and the inequality $1-e^{-s}\leq s$, $s\geq 0$ implies that
\[
\biggl\lvert \int_{-\infty}^0 \newaplus _n(s) (e^{\la s}-1)\,ds\biggr\rvert
\leq \la \int_{-\infty}^0 \newaplus _n(s) |s|\,ds\leq \la \int_\R \newaplus (s) |s|\,ds<\infty,
\]
by \eqref{firstmoment}. As a result, if we set
\[
\la_1:=(\kap  {A}^+_{R_1}-m)\biggl(\kap \int_\R \newaplus (s) |s|\,ds+|G_\xi(\la_*)|\biggr)^{-1}>0,
\]
then, for any $\la\in (0,\la_1)$, we have
\[
  \la G_\xi^{(n)}(\la)\geq \kap  {A}^+_{R_1}-m-\la_1\kap \int_\R \newaplus _n(s) |s|\,ds
  = \la_1 |G_\xi(\la_*)|\geq \la G_\xi(\la_*),
\]
i.e. $G_\xi^{(n)}(\la)\geq G_\xi(\la_*)=\inf\limits_{\la>0}G_\xi(\la)$. Then, for any $n\geq N$, \eqref{wqwqwqwqwe} implies that $\la_*^{(n)}$, being the minimum point for $G_\xi^{(n)}$, does not belong to the interval $(0,\la_1)$. Next, let $N_1\geq N$ be such that $R_n\geq \frac{1}{\la_1}$, for all $n\geq N_1$. Then, for any $\la\geq\la_1$, and for any $n>m\geq N_1$, we have $R_n>R_m$ and
\begin{align*}
\T_\xi^{(n)}(\la)-\T_\xi^{(m)}(\la)&=\kap \int_{R_m}^{R_n}(1-\la s)\newaplus (s)e^{\la s}\,ds\\ &\leq
\kap \int_{R_m}^{R_n}(1-\la_1 s)\newaplus (s)e^{\la s}\,ds\leq 0.
\end{align*}

As a result, the sequence $\{\la_*^{(n)}\mid n\geq N_1\}\subset[\la_1,\infty)$ is monotonically decreasing (cf.~\eqref{derofT}). We set
\begin{equation}\label{barla}
  \vartheta:=\lim_{n\to\infty}\la_*^{(n)}\geq\la_1.
\end{equation}
Next, for any $n,m\in\N$, $n>m\geq N_1$,
\begin{equation}\label{doubleineqcoolmy}
  G^{(n)}_\xi(\la_*^{(n)})\geq G^{(m)}_\xi(\la_*^{(n)})\geq
G^{(m)}_\xi(\la_*^{(m)}),
\end{equation}
where we used that $G^{(n)}_\xi$ is increasing in $n$ and $\la_*^{(m)}$ is the minimum point of $G^{(m)}_\xi$. Therefore, the sequence $\{G^{(n)}_\xi(\la_*^{(n)})\}$ is increasing and, by \eqref{wqwqwqwqwe}, is bounded. Then, there exists
\begin{equation}\label{limitgmy}
  \lim\limits_{n\to\infty}G^{(n)}_\xi(\la_*^{(n)})=:g \leq G_\xi(\la_*)-\delta.
\end{equation}
Fix $m\geq N_1$ in \eqref{doubleineqcoolmy} and pass $n$ to infinity; then, by the continuity of $G_\xi^{(n)}$,
\begin{equation}\label{wt3udh}
  g\geq \lim\limits_{\la\to\vartheta+}G_\xi^{(m)}(\la)=G_\xi^{(m)}(\vartheta)\geq
  G_\xi^{(m)}(\la^{(m)}),
\end{equation}
in particular, $\vartheta>0$, as $G_\xi^{(m)}(0+)=\infty$. Next, if we pass $m$ to $\infty$ in \eqref{wt3udh}, we will get from \eqref{limitgmy}
\begin{equation}\label{contradictionmy}
  \lim_{m\to\infty}G_\xi^{(m)}(\vartheta)=g\leq G_\xi(\la_*)-\delta<G_\xi(\la_*).
\end{equation}
If $0<\vartheta\leq\sigma(\newaplus )$ then
\[
\lim_{m\to\infty}G_\xi^{(m)}(\vartheta)=G_\xi(\vartheta)\geq G_\xi(\la_*),
\]
that contradicts \eqref{contradictionmy}. If $\vartheta>\sigma(\newaplus )$, then
$\lim\limits_{m\to\infty}G_\xi^{(m)}(\vartheta)=\infty$ (recall again that $\L^-(\newaplus )(\la)$ is analytic and $\L^-(\newaplus )(\la)$ is monotone in $\la$), that contradicts \eqref{contradictionmy} as well.

The contradiction we obtained shows that \eqref{infsareequal} does hold. Then, for the chosen $c\geq c_*(\xi)$ which satisfies \eqref{hypoteticspeed}, one can find $n$ big enough to ensure that, cf.~\eqref{minspeedfortrunc},
\[
c<\inf\limits_{\la>0}G^{(n)}_\xi(\la)=c_*^{(n)}(\xi).
\]
However, as it was shown above, for this $n$ there exists a profile $\psi_n$ of a traveling wave to the `truncated' equation \eqref{eq:basic_one_dim}  with $\newaplusminus$ replaced by $\newaplusminus_{R_n}$. The latter contradicts the statement of Theorem~\ref{thm:trwexists} applied to this equation, as $c_*^{(n)}(\xi)$ has to be a minimal possible speed for such waves.

Therefore, the strict inequality in \eqref{hypoteticspeed} is impossible, hence, we have equality in \eqref{ineqwillbeeq}. As a result, \eqref{aplusexpint1} and \eqref{eq:dedAxi} imply \eqref{minspeed}, and \eqref{dopequal} may be read as \eqref{minspeed-spec}. The rest of the statement is evident now.
\end{proof}

\begin{remark}\label{rem:evenimplypos}
  Clearly, the assumption $\aplus(-x)=\aplus(x)$, $x\in\X$, implies $\m_\xi=0$, for any $\xi\in S^{d-1}  $. As a result, all speeds of traveling waves in any directions are positive, by \eqref{minspeed}.
\end{remark}

\section{Asymptotic and uniqueness}\label{sec:asympt_uniq}
In this subsection we will prove the uniqueness (up to shifts) of a profile $\psi$ for a traveling wave with given speed $c\geq c_*(\xi)$, $c\neq0$.
We will use the almost traditional now approach, namely, we find an {\em a priori} asymptotic for $\psi(t)$, $t\to\infty$, cf. e.g. \cite{CC2004,AGT2012} and the references therein.

We start with the so-called characteristic function of the equation \eqref{eq:basic}. Namely, for a given $\xi\in S^{d-1}  $ and for any $c\in [c_*(\xi),\infty)$, we set
\begin{equation}\label{charfunc}
\h_{\xi,c}(z):= \kap (\L \newaplus )(z)-m-z c=zG_\xi(z)-zc, \qquad \Re z\in I_\xi .
\end{equation}
\begin{proposition}\label{prop:charfuncpropert}
Let $\xi\in S^{d-1}  $ be fixed, $\newaplus\in\Uxi$, $\sigmaplus :=\sigma(\newaplus )$, $c_*(\xi)$ be the minimal traveling wave speed in the direction $\xi$. Let, for any $c\geq c_*(\xi)$, the function $\psi\in\M$ be a traveling wave profile corresponding to the speed~$c$. For the case $\newaplus\in\Wxi$ with $m=\T_\xi(\sigmaplus )$, we will assume, additionally, that
\begin{equation}\label{secmomadd}
  \int_\R s^2 \newaplus (s)e^{\sigmaplus  s}\,ds<\infty.
\end{equation}
Then the function $\h_{\xi,c}$ is analytic on $\{0<\Re z<\sigma(\psi)\}$. Moreover, for any $\beta\in(0,\sigma(\psi))$, the function $\h_{\xi,c}$ is continuous and does not equal to $0$ on the closed strip $\{\beta\leq\Re z\leq\sigma(\psi)\}$, except the root at $z=\sigma(\psi)$, whose multiplicity $j$ may be $1$ or $2$ only.
\end{proposition}
\begin{proof}
By \eqref{neweq} and the arguments around, $\h_{\xi,c}(z)=z (G_\xi(z) -c)$ is analytic on $\{0<\Re z<\sigma(\psi)\}\subset I_\xi$ and does not equal to $0$ there. Then, by \eqref{speedviaabs} and Proposition~\ref{prop:infGisreached}, the smallest positive root of the function $\h_{\xi,c}(\la)$ on $\R$ is exactly $\sigma(\psi)$. Prove that if $z_0:=\sigma(\psi)+i \beta$ is a root of $\h_{\xi,c}$, then $\beta=0$. Indeed,  $\h_{\xi,c}(z_0)=0$ yields
\[
\kap \int_\R \newaplus (s) e^{\sigma(\psi) s}\cos\beta s \,ds=m+c\sigma(\psi),
\]
that together with \eqref{speedviaabs} leads to
\[
\kap \int_\R \newaplus (s) e^{\sigma(\psi) s}(\cos\beta s -1)\,ds=0,
\]
and thus $\beta=0$.

Regarding multiplicity of the root $z=\sigma(\psi)$, we note that, by Proposition~\ref{prop:infGisreached} and Corollary~\ref{cor_alltogether}, there exist two possibilities. If $\newaplus\in\Vxi$, then $\sigma(\psi)\leq\la_*<\sigma(\newaplus )$ and, therefore, $G_\xi$ is analytic at $z=\sigma(\psi)$. By the second equality in \eqref{charfunc}, the multiplicity $j$ of this root for $\h_{\xi,c}$ is the same as for the function $G_\xi(z)-c$. By Proposition~\ref{prop:infGisreached}, $G_\xi$ is strictly decreasing on $(0,\la_*)$ and, therefore, $j=1$ for $c>c_*(\xi)$. By Corollary~\ref{cor_alltogether}, for $c=c_*(\xi)$, we have $G_\xi'(\sigma(\psi))=G_\xi'(\la_*)=0$ and, since $\h_{\xi,c}''(\sigmaplus )>0$, one gets $j=2$.

Let now $\newaplus\in\Wxi$. Then, we recall, $\la_*=\sigmaplus :=\sigma(\newaplus )<\infty$, $G_\xi(\sigmaplus )<\infty$ and \eqref{leftlimG} hold.
For $c>c_*(\xi)$, the arguments are the same as before, and they yield $j=1$. Let $c=c_*(\xi)$. Then $\h_{\xi,c}(\sigmaplus )=0$, and,  for all $z\in\mathbb{C}$, $\Re z\in(0,{\sigmaplus })$, one has
\begin{align}
 \h_{\xi,c}({\sigmaplus }-z)&=\h_{\xi,c}({\sigmaplus }-z)-\h_{\xi,c}({\sigmaplus })=\kap \int_\R \newaplus (\tau )(e^{({\sigmaplus }-z)\tau }-e^{{\sigmaplus } \tau })d\tau +cz \notag\\
    &=z\biggl(-\kap \int_\R \newaplus (\tau )e^{{\sigmaplus } \tau } \int_0^\tau e^{-zs}\,ds d\tau +c
    \biggr).\label{eq:i}
\end{align}
Let $z=\alpha+\beta i$, $\alpha\in(0,\sigmaplus )$. Then $\bigl\lvert e^{\sigmaplus  \tau }  e^{-zs}\bigr\rvert=e^{\sigmaplus \tau-\alpha s}$. Next, for $\tau\geq0$, $s\in[0,\tau]$, we have
$e^{\sigmaplus \tau-\alpha s}\leq e^{\sigmaplus \tau}$; whereas, for $\tau<0$, $s\in[\tau,0]$, one has $e^{\sigmaplus \tau-\alpha s}=e^{\sigma(\tau-s)}e^{(\sigmaplus -\alpha) s}\leq 1$.
As a result, $\bigl\lvert e^{\sigmaplus  \tau }  e^{-zs}\bigr\rvert\leq e^{\sigmaplus \max\{\tau,0\}}$.
Then, using that $\newaplus\in\Wxi$ implies $\int_\R\newaplus (\tau)e^{\sigmaplus \max\{\tau,0\}}\,ds<\infty$, one can apply the dominated convergence theorem to the double integral in \eqref{eq:i}; we get then
\begin{equation}\label{impeq}
  \lim_{\substack{\Re z\to0+\\ \Im z\to0}}\frac{\h_{\xi,c}({\sigmaplus }-z)}{z}
  =-\kap \int_\R \newaplus (\tau )e^{{\sigmaplus } \tau } \tau d\tau +c.
\end{equation}
According to the statement 3 of Theorem~\ref{thm:speedandprofile}, for $m<\T_\xi(\sigmaplus )$, the r.h.s. of \eqref{impeq} is positive, i.e. $j=1$ in such a case. Let now $m=\T_\xi(\sigmaplus )$, then the r.h.s. of \eqref{impeq} is equal to $0$. It is easily seen that one can rewrite then \eqref{eq:i} as follows
\begin{align}\notag
  \frac{\h_{\xi,c}({\sigmaplus }-z)}{z}&=\kap \int_\R \newaplus (\tau )e^{{\sigmaplus } \tau } \int_0^\tau (1-e^{-zs})\,ds d\tau\\\label{eq:ii}
  &=z\kap \int_\R \newaplus (\tau )e^{{\sigmaplus } \tau } \int_0^\tau \int_0^s e^{-zt}\,dt\,ds\, d\tau.
\end{align}
Similarly to the above, for $\Re z\in (0,\sigmaplus )$, one has that
$\lvert e^{{\sigmaplus } \tau -zt}\rvert \leq e^{{\sigmaplus } \max\{\tau,0\} }$.
Then, by \eqref{secmomadd} and the dominated convergence theorem, we get from \eqref{eq:ii} that
\begin{equation*}
  \lim_{\substack{\Re z\to0+\\ \Im z\to0}}\frac{\h_{\xi,c}({\sigmaplus }-z)}{z^2}
  =\frac{\kap }{2}\int_\R \newaplus (\tau )e^{{\sigmaplus } \tau } \tau^2 d\tau\in(0,\infty).
\end{equation*}
Thus $j=2$ in such a case. The statement is fully proved now.
\end{proof}

\begin{remark}
Combining results of Theorem~\ref{thm:speedandprofile} and Proposition~\ref{prop:charfuncpropert}, we immediately get that, for the case $j=2$, the minimal traveling wave speed $c_*(\xi)$ always satisfies \eqref{minspeed-spec}.
\end{remark}

\begin{remark}\label{very-critical-case}
  If $\newaplus $ is given by \eqref{exofa}, then, cf.~Example~\ref{ex:spefunc},
  the case $\newaplus\in\Wxi$, $m=\T_\xi(\sigmaplus )$ together with \eqref{secmomadd} requires $p=1$, $\mu<\mu_*$, $q>3$.
\end{remark}

In order to include the critical case $\sigma(\newaplus)=\sigma(\psi_{c_*})$, we consider the following analogue of the Ikehara complex Tauberian theorem, cf.~\cite{Ten1995,Del1954,Kab2008}.
Let, for any $D\subset\mathbb{C}$, $\An(D)$ be the class of all holomorphic functions on $D$.
\begin{proposition}[{{\cite[Theorem 2]{FKT100-4}}}]\label{prop:tauber}
Let $\varphi:\R_+\to  \R_+:=[0,\infty)$ be a non-increasing function such that, for some $\mu>0$, $\nu>0$,
\begin{equation}\label{assum:1}
    \text{the function } e^{\nu t} \varphi(t) \text{ is non-decreasing,}
\end{equation}
and 
\begin{equation}\label{eq:psi_onesided_lap}
    \int\limits_0^{\infty}e^{z t}d\varphi(t)<\infty, \quad 0< \Re z< \mu. 
\end{equation}
Let also the following assumptions hold.
\begin{enumerate}
  \item There exist a constant $j>0$ and complex-valued functions 
    \[
      H\in\An(0<\Re z\leq \mu), \qquad  F\in\An (0<\Re z<\mu)\cap C(0<\Re z\leq \mu),
    \]
    such that the following representation holds
    \begin{equation}\label{eq:sing_repres}
      \int\limits_0^{\infty}e^{z t}\varphi(t)dt=\dfrac{F(z)}{(\mu-z)^j}+H(z),\quad 0<\Re z<\mu.
    \end{equation}
  \item For any $T>0$,
    \begin{equation}\label{loglim}
      \lim_{\sigma\to0+}q_j(\sigma)\sup_{|\tau|\leq T}\bigl\lvert F(\mu-2\sigma-i\tau)-F(\mu-\sigma-i\tau)\bigr\rvert=0,
    \end{equation}
  where, for $\sigma>0$,
  \begin{equation}\label{eq:def_gj}
      q_j(\sigma):=\begin{cases}
    \sigma^{j-1}, & 0<j<1,\\
    \log \sigma, & j=1,\\
    1, &j>1.
    \end{cases}
  \end{equation}
\end{enumerate}
Then $\varphi$ has the following asymptotic
\begin{equation}\label{eq:psi_asympt}
  \varphi(t)\sim \frac{F(\mu)}{\Gamma(j)} t^{j-1}e^{-\mu t} ,\quad t\to \infty.
\end{equation}
\end{proposition}

Now, we can apply Proposition~\ref{prop:tauber} to find the asymptotic of the profile of a traveling wave.
\begin{proposition}\label{asymptexact}
In conditions and notations of Proposition~\ref{prop:charfuncpropert},
for $c\neq0$, there exists $D=D_j>0$, such that
\begin{equation}\label{eq:trw_asympt}
\psi(t)\sim D e^{-\sigma(\psi) t}t^{j-1},\quad t\to  \infty.
\end{equation}
\end{proposition}
\begin{proof}
We set $\mu:=\sigma(\psi)$ and
\begin{equation}\label{deffunc}
\begin{aligned}
	f(z) &:= \kl\bigl( \L(\psi^2) \bigr)(z) + \kn\bigl(\L(\psi (\newaminus *\psi))\bigr)(z), & g_j(z)&:=\dfrac{\h_{\xi,c}(z)}{(z-\mu)^j},\\
	H(z) &:= -\int\limits_{-\infty}^0\psi(t)e^{z t}dt, & F(z)&:=\dfrac{f(z)}{g_j(z)}.
\end{aligned}
\end{equation}
For any $\mu>\beta>0$, $T>0$, we set
\[
K_{\beta,\mu,T}:=\bigl\{z\in\mathbb{C} \bigm| \beta\leq \Re z \leq \mu,\ \lvert \Im z \rvert \leq T\bigr\}.
\]
By \eqref{unexpineq} and Lemma~\ref{lem:allaboutLapl},  we have that $f, H\in\An(0<\Re z\leq \mu)$; in particular, for any $T>0$, $\beta>0$,
\begin{equation}\label{fbdd}
\bar f:=\sup_{z\in K_{\beta,\mu,T}}|f(z)|<\infty.
\end{equation}
By Proposition~\ref{prop:charfuncpropert}, the function $g_j$ is continuous and does not equal to $0$ on the strip $\{0<\Re z\leq\mu\}$, in particular, for any $T>0$, $\beta>0$,
\begin{equation}\label{gpos}
\bar g_j:=\inf_{z\in K_{\beta,\mu,T}}|g(z)|>0.
\end{equation}
Therefore, $F\in\An (0<\Re z<\mu)\cap C(0< \Re z \leq \mu)$. As a result, one can rewrite \eqref{rewriting} in the form \eqref{eq:sing_repres}, with $\varphi=\psi$ and with $F$, $H$ as in \eqref{deffunc}.

Taking into account the forth statement of Theorem~\ref{thm:trwexists}, to apply Proposition~\ref{prop:tauber} it is enough to prove that \eqref{loglim} holds. Assume that $0<2\sigma<\mu$.

Let $j=2$. Clearly, $F\in C(0< \Re z \leq \mu)$ implies that $F$ is uniformly continuous
on $ K_{\beta,\mu,T}$. Then, for any $\eps>0$ there exists $\delta>0$ such that, for any $\tau\in[-T,T]$, the inequality
\[
	|\sigma|=|(\mu-2\sigma-i\tau)-(\mu-\sigma-i\tau)|<\delta,
\]
implies
\[
	|F(\mu-2\sigma-i\tau)-F(\mu-\sigma-i\tau)|<\eps,
\]
and hence \eqref{loglim} holds (with $j=2$).

Let now $j=1$. If $F\in \An(K_{\beta,\mu,T})$, we have, evidently, that $F'$ is bounded on $K_{\beta,\mu,T}$, and one can apply a mean-value-type theorem for complex-valued functions, see e.g. \cite{EJ1992}, to get that $F$ is a Lipschitz function on $K_{\beta,\mu,T}$. Therefore, for some $K>0$,
\[
|F(\mu-2\sigma-i\tau)-F(\mu-\sigma-i\tau)|<K |\sigma|,
\]
for all $\tau\in[-T,T]$, that yields \eqref{loglim} (with $j=1$). By~Proposition~\ref{prop:infGisreached} and Corollary~\ref{cor_alltogether}, the inclusion $F\in \An(K_{\beta,\mu,T})$ always holds for $c>c_*$; whereas, for $c=c_*$ it does hold iff $\newaplus\in\Vxi$. Moreover, the case $\newaplus\in\Wxi$ with $m=\T_\xi(\sigmaplus )$ and $c=c_*$ implies, by Proposition~\ref{prop:charfuncpropert}, $j=2$ and hence it was considered above.

Therefore, it remains to prove \eqref{loglim} for the case $\newaplus\in\Wxi$ with $m<\T_\xi(\sigmaplus )$, $c=c_*$ (then $j=1$).
Denote, for simplicity,
\begin{equation}\label{zi}
z_1:=\mu-\sigma-i\tau, \qquad z_2:=\mu-2\sigma-i\tau.
\end{equation}
Then, by \eqref{deffunc}, \eqref{fbdd}, \eqref{gpos}, one has
\begin{align}
  \bigl\lvert F(z_2)-F(z_1)\bigr\rvert &\leq
  \Bigl\lvert \frac{f(z_2)}{g_1(z_2)}-
  \frac{f(z_1)}{g_1(z_2)}\Bigr\rvert+\Bigl\lvert
  \frac{f(z_1)}{g_1(z_2)}
  -\frac{f(z_1)}{g_1(z_1)}\Bigr\rvert\notag\\
  &\leq \frac{1}{\bar g_1}\bigl\lvert f(z_2)-
  f(z_1)\bigr\rvert+ \frac{\bar f}{\bar g_1^2 }|g_1(z_1)-g_1(z_2)|.\label{fgest}
\end{align}
Note that, if $0<\phi\in L^\infty(\R)\cap L^1(\R)$ be such that $\sigma(\phi)>\mu$ then
\begin{align}
&  \bigl\lvert (\L\phi)(z_2)-
  (\L\phi)(z_1)\bigr\rvert  \leq \int_\R \phi(s)e^{\mu s}|e^{-2\sigma s}-e^{-\sigma s}|ds\notag\\
   &\quad \leq\sigma \int_0^\infty \phi(s)e^{(\mu-\sigma) s}sds+
   \sigma \int_{-\infty}^0 \phi(s)e^{(\mu -2\sigma)s} |s|\,ds=O(\sigma), \label{genconv}
\end{align}
as $\sigma\to0+$, where we used that $\sup_{s<0}e^{(\mu -2\sigma)s} |s|<\infty$,  $0<2\sigma<\mu$, and that \ref{L-analytic} holds.
Applying \eqref{genconv} to $\phi=\psi (\newaminus *\psi)\leq \theta^2 \newaminus \in L^1(\R)\cap L^\infty(\R)$, one gets
\[
  \sup_{\tau\in[-T,T]}\bigl\lvert f(z_2)-
  f(z_1)\bigr\rvert  = O(\sigma), \quad \sigma\to0+.
\]
Therefore, by \eqref{fgest}, it remains to show that
\begin{equation}\label{enough}
  \lim_{\sigma\to0+} \log \sigma \sup_{\tau\in[-T,T]} |g_1(z_1)-g_1(z_2)|=0.
\end{equation}
Recall that, in the considered case $c=c_*$, one has $\h_{\xi,c}(\mu)=0$. Therefore, by \eqref{charfunc}, \eqref{deffunc}, \eqref{zi}, we have
\begin{align}
&\quad|g_1(z_1)-g_1(z_2)|=\biggl\lvert \frac{\h_{\xi,c}(z_1)-\h_{\xi,c}(\mu)}{z_1-\mu}
-\frac{\h_{\xi,c}(z_2)-\h_{\xi,c}(\mu)}{z_2-\mu}\biggr\rvert\notag\\&=
\biggl\lvert \frac{\kap (\L\newaplus )(z_1)-\kap (\L\newaplus )(\mu)}{z_1-\mu}
-\frac{\kap (\L\newaplus )(z_2)-\kap (\L\newaplus )(\mu)}{z_2-\mu}\biggr\rvert\notag\\
&\leq \kap \int_\R \newaplus (s) e^{\mu s}
\biggl\lvert \frac{1-e^{(-\sigma-i\tau)s}}{\sigma+i\tau}
-\frac{1-e^{(-2\sigma-i\tau)s}}{2\sigma+i\tau}\biggr\rvert\,ds\notag\\
&= \kap \int_\R \newaplus (s) e^{\mu s}\biggl\lvert\int_0^s \bigl( e^{(-\sigma-i\tau)t}-
e^{(-2\sigma-i\tau)t}\bigr)\,dt\biggr\rvert\,ds\notag\\&\leq
\kap \int_0^\infty \newaplus (s) e^{\mu s}\int_0^s \bigl| e^{-\sigma t}-
e^{-2\sigma t}\bigr|\,dt\,ds\notag
\\&\quad+\kap \int_{-\infty}^0 \newaplus (s) e^{\mu s}\int_s^0 \bigl| e^{-\sigma t}-
e^{-2\sigma t}\bigr|\,dt\,ds\label{forcont}\\
\intertext{and since, for $t\geq0$, $\bigl| e^{-\sigma t}-
e^{-2\sigma t}\bigr|\leq \sigma t$; and, for $s\leq t\leq0$,
\[
\bigl| e^{-\sigma t}- e^{-2\sigma t}\bigr|=e^{-2\sigma t}\bigl| e^{\sigma t}-
1\bigr|\leq e^{-2\sigma s} \sigma |t|,
\]
one can continue \eqref{forcont}}
&\leq \frac{1}{2}\sigma\kap \int_0^\infty \newaplus (s) e^{\mu s}s^2\,ds
+\frac{1}{2}\sigma\kap \int_{-\infty}^0 \newaplus (s) e^{(\mu-2\sigma) s}s^2\, ds . \notag
\end{align}
Since $\mu>2\sigma$, one has $\sup\limits_{s\leq0}e^{(\mu-2\sigma) s}s^2<\infty$, therefore, by \eqref{secmomadd}, one gets
\[
\sup_{\tau\in[-T,T]}|g_1(z_1)-g_1(z_2)|\leq\mathrm{const}\cdot\sigma,
\]
that proves \eqref{enough}. The statement is fully proved now.
\end{proof}

\begin{remark}\label{rem:shifting}
  By \eqref{eq:psi_asympt} and \eqref{deffunc}, one has that the constant $D=D_j$ in \eqref{eq:trw_asympt} is given by
  \[
  	D=D(\psi) = \bigl( \kl\bigl( \L(\psi^2) \bigr)(\mu)+ \kn\bigl(\L(\psi (\newaminus *\psi))\bigr)(\mu)\bigr) \lim_{z\to\mu}\dfrac{(z-\mu)^j}{\h_{\xi,c}(z)},
  \]
  where $\mu=\sigma(\psi)$. Note that, by Proposition~\ref{prop:charfuncpropert}, the limit above is finite and does not depend on $\psi$. Next, by Remark~\ref{shiftoftrw}, for any $q\in\R$, $\psi_q(s):=\psi(s+q)$, $s\in\R$ is a traveling wave with the same speed, and hence, by Theorem~\ref{thm:speedandprofile}, $\sigma(\psi_q)=\sigma(\psi)$. Moreover,
  \begin{align*}
  \bigl(\L(\psi_q (\newaminus *\psi_q))\bigr)(\mu)&=\int_\R \psi(s+q)\int_\R \newaminus (t)\psi(s-t+q)\,dt\, e^{\mu s}\,ds\\&=e^{-\mu q}\bigl(\L(\psi (\newaminus *\psi))\bigr)(\mu),\\
	\bigl( \L(\psi_q^2) \bigr)(\mu) &= \int_\R \psi^2(s+q) e^{\mu s} ds = e^{-\mu q} \bigl( \L(\psi^2)\bigr) (\mu).
  \end{align*}
  Thus, for a traveling wave profile $\psi$ one can always choose a $q\in\R$ such that, for the shifted profile $\psi_q$, the corresponding $D=D(\psi_q)$ will be equal to $1$.
\end{remark}

Finally, we are ready to prove the uniqueness result.

\begin{theorem}\label{thm:tr_w_uniq}
Let $\xi\in S^{d-1}  $ be fixed and $\newaplus\in\Uxi$. Suppose, additionally, that \eqref{as:aplus-aminus-is-pos} holds. Let $c_*(\xi)$ be the minimal traveling wave speed according to Theorem~\ref{thm:trwexists}.
For the case $\newaplus\in\Wxi$ with $m=\T_\xi(\sigmaplus )$, we will assume, additionally, that
\eqref{secmomadd} holds.
Then, for any $c\geq c_*$, such that $c\neq0$, there exists a unique, up~to~a~shift, traveling wave profile $\psi$ for \eqref{eq:basic}.
\end{theorem}
\begin{proof}
  We will follow the sliding technique from \cite{CDM2008}. Let $\psi_1,\psi_2\in C^1(\R)\cap\M$ are traveling wave profiles with a speed $c\geq c_*$, $c\neq0$, cf.~Theorem~\ref{thm:trwexists}. By~Proposition~\ref{asymptexact} and Remark~\ref{rem:shifting}, we may assume, without lost of generality, that \eqref{eq:trw_asympt} holds for both $\psi_1$ and $\psi_2$ with $D=1$. By~the~proof of Proposition~\ref{prop:charfuncpropert}, the corresponding $j\in\{1,2\}$ depends on $\newaplusminus$, $\kapm$, $m$ only, and does not depend on the choice of $\psi_1$, $\psi_2$. By~Theorem~\ref{thm:speedandprofile}, $\sigma(\psi_1)=\sigma(\psi_2)=:\la_c\in (0,\infty)$.

\smallskip

  \textit{Step 1.} Prove that, for any $\tau>0$, there exists $T=T(\tau)>0$, such that
  \begin{equation}\label{ineqoutoftau}
    \psi^\tau_1(s):=\psi_1(s-\tau)>\psi_2(s), \quad s\geq T.
  \end{equation}

  Indeed, take an arbitrary $\tau>0$. Then \eqref{eq:trw_asympt} with $D=1$ yields
  \[
    \lim_{s\to\infty}\frac{\psi_1^\tau(s)}{(s-\tau)^{j-1}e^{-\la_c(s-\tau)}}=1=
    \lim_{s\to\infty}\frac{\psi_2(s)}{s^{j-1}e^{-\la_c s}}.
  \]
  Then, for any $\eps>0$, there exists $T_1=T_1(\eps)>\tau$, such that, for any $s>T_1$,
  \[
  \frac{\psi_1^\tau(s)}{(s-\tau)^{j-1}e^{-\la_c(s-\tau)}}-1>-\eps, \qquad
  \frac{\psi_2(s)}{s^{j-1}e^{-\la_c s}}-1<\eps.
  \]
  As a result, for $s>T_1>\tau$,
  \begin{align}
  &\quad \psi_1^\tau(s)-\psi_2(s)>(1-\eps)(s-\tau)^{j-1}e^{-\la_c(s-\tau)} -(1+\eps)s^{j-1}e^{-\la_c s}\notag\\
  &=s^{j-1}e^{-\la_c s}\biggl( \Bigl(1-\frac{\tau}{s}\Bigr)^{j-1}e^{\la_c \tau}-1 -\eps\Bigl(\Bigl(1-\frac{\tau}{s}\Bigr)^{j-1}e^{\la_c \tau} +1\Bigr)\biggr)\notag\\
  &\geq s^{j-1}e^{-\la_c s}\biggl( \Bigl(1-\frac{\tau}{T_1}\Bigr)^{j-1}e^{\la_c \tau}-1 -\eps\bigl(e^{\la_c \tau} +1\bigr)\biggr)>0,  \label{poswillbe}
  \end{align}
  if only
  \begin{equation}\label{smallepsmy}
      0<\eps<\frac{\Bigl(1-\dfrac{\tau}{T_1}\Bigr)^{j-1}e^{\la_c \tau}-1 }{e^{\la_c \tau} +1}=:g(\tau,T_1).
  \end{equation}

  For $j=1$, the nominator in the r.h.s. of \eqref{smallepsmy} is positive. For $j=2$, consider $f(t):= \bigl(1-\frac{t}{T_1}\bigr)e^{\la_c t}-1$, $t\geq0$. Then
 $f'(t)=\frac{1}{T_1}e^{\la_c t}(\la_c T_1-\la_c t-1)>0$, if~only $T_1>t+\frac{1}{\la_c}$, that implies $f(t)>f(0)=0$, $t\in\bigl(0,T_1-\frac{1}{\la_c}\bigr)$.

 As a result, choose $\eps=\eps(\tau)>0$ with $\eps< g\bigl(\tau,\tau+\frac{1}{\la_c}\bigr)$, then, without loss of generality, suppose that $T_1=T_1(\eps)=T_1(\tau)>\tau+\frac{1}{\la_c}>\tau$. Therefore, $0<\eps<g\bigl(\tau,\tau+\frac{1}{\la_c}\bigr)\leq g(\tau,T_1)$, that fulfills \eqref{smallepsmy}, and hence
\eqref{poswillbe} yields \eqref{ineqoutoftau}, with any $T>T_1$.

\smallskip

\textit{Step 2.} Prove that there exists $\nu >0$, such that, cf.~\eqref{ineqoutoftau},
  \begin{equation}\label{ineqeverywhere}
    \psi^\nu _1(s) \geq \psi_2(s), \quad s\in\R.
  \end{equation}

  Let $\tau>0$ be arbitrary and $T=T(\tau)$ be as above.
  Choose any $\delta\in\bigl(0,\frac{\theta}{4}\bigr)$.
  By  \eqref{eq:deftrw}, and the dominated convergence theorem,
\begin{equation}\label{convtotthetaconv}
    \lim\limits_{s\to-\infty} (\newaminus *\psi_2)(s)=\lim\limits_{s\to-\infty} \int_\R \newaminus (\tau)\psi_2(s-\tau)\,d\tau=\theta>\delta.
\end{equation}
  Then, one can choose $T_2=T_2(\delta)>T$, such that, for all $s<-T_2$,
  \begin{gather}\label{bigT2}
    \psi_1^\tau(s)>\theta-\delta,\\
    \kl\psi_2(s)+\kn(\newaminus *\psi_2)(s)>\delta.\label{bigT2conv}
  \end{gather}
  Note also that \eqref{ineqoutoftau} holds, for all $s\geq T_2>T$, as well.
  Clearly, for any $\nu\geq\tau$,
  \[
  \psi_1^\nu (s)=\psi_1(s-\nu )\geq\psi_1(s-\tau)>\psi_2(s), \quad s>T_2.
  \]
  Next, $\lim\limits_{\nu \to\infty}\psi_1^\nu (T_2)=\theta>\psi_2(-T_2)$ implies that there exists $\nu _1=\nu _1(T_2)=\nu _1(\delta)>\tau$, such that, for all $\nu >\nu _1$,
  \[
    \psi_1^\nu (s)\geq \psi_1^\nu (T_2)>\psi_2(-T_2)\geq \psi_2(s), \quad s\in[-T_2,T_2].
  \]
  Let such a $\nu >\nu _1$ be chosen and fixed.
  As a result,
  \begin{align}
  \psi_1^\nu (s)\geq \psi_2(s), \quad s\geq -T_2, \label{ineqright}\\
  \intertext{and, by \eqref{bigT2},}
  \psi_1^\nu (s)+\delta>\theta>\psi_2(s), \quad s<-T_2. \label{ineqleft}
  \end{align}
  For the $\nu >\nu _1$ chosen above, define
  \begin{equation}\label{phinu}
   \varphi_\nu (s):=\psi_1^\nu (s)-\psi_2(s), \quad s\in\R.
  \end{equation}
  To prove \eqref{ineqeverywhere}, it is enough to show that $\varphi_\nu (s)\geq0$, $s\in\R$.

  On the contrary, suppose that $\varphi_\nu $ takes negative values. By~\eqref{ineqright}, \eqref{ineqleft},
  \begin{equation}\label{varnubelow}
      \varphi_\nu (s)\geq-\delta, \quad s<-T_2; \qquad \varphi_\nu (s)\geq 0,\quad s\geq-T_2.
  \end{equation}
  Since $\lim\limits_{s\to-\infty} \varphi_\nu (s)=0$ and $\varphi_\nu \in C^1(\R)$, our assumption implies that there exists $s_0<-T_2$, such that
  \begin{equation}\label{minpointmy}
      \varphi_\nu (s_0)=\min\limits_{s\in\R}\varphi_\nu (s)\in [-\delta,0).
  \end{equation}
  We set also
  \begin{equation}\label{deltamy}
    \delta_*:=-\varphi_\nu (s_0)=\psi_2(s_0)-\psi_1^\nu(s_0)\in (0,\delta].
  \end{equation}

Next, both $\psi_1^\nu $ and $\psi_2$ solve \eqref{eq:trw}.
By \eqref{eq:defJtheta}, 
$\int_\R J_\theta(s)\,ds=\kap-\kn\theta$. Denote $L_\theta \varphi:=J_\theta *\varphi-(\kap{-}\kn\theta)\varphi$. Then one can rewrite \eqref{eq:trw} 
\[
  c\psi'(s) + ({L}_\theta\psi)(s) + (\theta-\psi(s)) \bigl(\kl\psi(s) + \kn(\newaminus *\psi)(s)\bigr)=0.
\]
Writing the latter equation for $\psi_1^\nu$ and $\psi_2$ and subtracting the results, one gets
\begin{equation}\label{neweqnpos}
  \begin{gathered}
  c\varphi_\nu'(s)+(L_\theta\varphi_\nu)(s)+A(s)=0,\\
  A(s):=(\theta-\psi_1^\nu(s)) \bigl(\kl\psi_1^\nu(s) + \kn(\newaminus *\psi_1^\nu)(s)\bigr) \\
			\qquad \qquad \qquad -(\theta-\psi_2(s)) \bigl(\kl\psi_2(s) + \kn(\newaminus *\psi_2)(s)\bigr).
  \end{gathered}
\end{equation}
Consider \eqref{neweqnpos} at the point $s_0$. By~\eqref{minpointmy},
\begin{equation}\label{maxprinciple}
\varphi_\nu'(s_0)=0, \qquad (L_\theta\varphi_\nu)(s_0)\geq0.
\end{equation}
 Next, \eqref{deltamy} yields
\begin{align}
  A(s_0)  = &
  (\theta-\psi_1^\nu(s_0)) \bigl(\kl\psi_1^\nu(s_0) + \kn(\newaminus *\psi_1^\nu)(s_0)\bigr) \notag\\
		&+ (\delta_*-(\theta-\psi_1^\nu(s_0)) \bigl(\kl\psi_2(s_0) + \kn(\newaminus *\psi_2)(s_0)\bigr) \notag\\
  =& (\theta-\psi_1^\nu(s_0)) \bigl(\kl\varphi_\nu(s_0) + \kn(\newaminus *\varphi_\nu)(s_0)\bigr) \notag\\
		&+\delta_* \bigl(\kl\psi_2(s_0)+\kn(\newaminus *\psi_2)(s_0)\bigr) \notag\\
  =& (\theta-\psi_1^\nu(s_0)) \bigl( \kl\varphi_\nu(s_0) + \kn (\newaminus *(\varphi_\nu+\delta_*))(s_0)\bigr) \notag\\
  	&+\delta_*\bigl( \kl\psi_2(s_0) + \kn(\newaminus *\psi_2)(s_0) -(\theta-\psi_1^\nu(s_0))\bigr) \notag\\
	>&0,\label{ufffffffff}
\end{align}
because of \eqref{minpointmy}, \eqref{bigT2}, and \eqref{bigT2conv}. The strict inequality in \eqref{ufffffffff} together with \eqref{maxprinciple} contradict to~\eqref{neweqnpos}. Therefore, \eqref{ineqeverywhere} holds, for any $\nu>\nu_1$.

\smallskip

\textit{Step 3.} Prove that, cf.~\eqref{ineqeverywhere},
  \begin{equation}\label{infiszero}
    \vartheta_*:=\inf\{\vartheta>0\mid \psi^\vartheta _1(s) \geq \psi_2(s), s\in\R\}=0.
  \end{equation}

On the contrary, suppose that $\vartheta_*>0$. Let $\varphi_*:=\varphi_{\vartheta_*}$ be given by \eqref{phinu}. By the continuity of the profiles, $\varphi_*\geq0$.

First, assume that $\varphi_*(s_0)=0$, for some $s_0\in\R$, i.e. $\varphi_*$ attains its minimum at $s_0$. Then \eqref{maxprinciple} holds with $\vartheta$ replaced by $\vartheta_*$, and, moreover, cf.~\eqref{neweqnpos},
\[
	A(s_0) = \kn(\theta-\psi_1^\vartheta(s_0))(\newaminus *\varphi_*)(s_0)\geq0.
\]
Therefore, \eqref{neweqnpos} implies
\begin{equation}\label{maxprLtheta}
  (L_\theta\varphi_*)(s_0)=0.
\end{equation}
By the same arguments as in the proof of Proposition~\ref{prop:infGisreached}, one can show that \eqref{as:aplus-aminus-is-pos} implies that the function $J_\theta$ also satisfies \eqref{as:aplus-aminus-is-pos}, for $d=1$, with some another constants.
Then, arguing in the same way as in the proof of \cite[Proposition 5.2]{FT2017a} (with $d=1$ and $\newaplus$ replaced by $J_\theta$), one gets that \eqref{maxprLtheta} implies
that $\varphi_*$ is a constant, and thus $\varphi_*\equiv0$, i.e. $\psi_1^{\vartheta_*}\equiv\psi_2$. The latter contradicts \eqref{ineqoutoftau}.

Therefore, $\varphi_*(s)>0$, i.e. $\psi_1^{\vartheta_*}(s)>\psi_2(s)$, $s\in\R$. By~\eqref{ineqoutoftau} and \eqref{convtotthetaconv},  there exists $T_3=T_3(\vartheta_*)>0$, such that
$\psi_1^{\frac{\vartheta_*}{2}}(s)>\psi_2(s)$, $s>T_3$, and also, for any $s<-T_3$,
\eqref{bigT2conv} holds and \eqref{ineqleft} holds with $\vartheta$ replaced by $\frac{\vartheta_*}{2}$ (for some fixed $\delta\in\bigl(0,\frac{\theta}{4}\bigr)$).
For any $\eps\in\bigl(0,\frac{\vartheta_*}{2}\bigr)$, $\psi_1^{\vartheta_*-\eps}\geq\psi_1^{\frac{\vartheta_*}{2}}$, therefore,
\[
\psi_1^{\vartheta_*-\eps}(s)>\psi_2(s), \quad s>T_3,
\]
and also \eqref{ineqleft} holds with $\vartheta$ replaced by $\vartheta_*-\eps$, for $s<-T_3$.
We set
\[
\alpha:=\inf\limits_{t\in[-T_3,T_3]}(\psi_1^{\vartheta_*}(s)-\psi_2(s))>0.
\]
Since the family $\bigl\{\psi_1^{\vartheta_*-\eps}\mid \eps\in\bigl(0,\frac{\vartheta_*}{2}\bigr)\bigr\}$ is monotone in $\eps$,
and $\lim\limits_{\eps\to0}\psi_1^{\vartheta_*-\eps}(t)=
\psi_1^{\vartheta_*}(t)$, $t\in\R$, we have, by Dini's theorem, that the latter convergence is uniform on $[-T_3,T_3]$. As~a~result, there exists $\eps=\eps(\alpha)\in \bigl(0,\frac{\vartheta_*}{2}\bigr)$, such that
\[
\psi_1^{\vartheta_*}(s)\geq\psi_1^{\vartheta_*-\eps}(s)\geq\psi_2(s), \quad s\in[-T_3,T_3].
\]
Then, the same arguments as in the Step~2 prove that $\psi_1^{\vartheta_*-\eps}(s)\geq\psi_2(s)$, for all $s\in\R$, that contradicts the definition \eqref{infiszero} of $\vartheta_*$.

As a result, $\vartheta_*=0$, and by the continuity of profiles, $\psi_1\geq\psi_2$. By the same arguments, $\psi_2\geq\psi_1$, that fulfills the statement.
\end{proof}

\section*{Acknowledgments}
Authors gratefully acknowledge the financial support by the DFG through CRC 701 ``Stochastic
Dynamics: Mathematical Theory and Applications'' (DF, YK, PT), the European
Commission under the project STREVCOMS PIRSES-2013-612669 (DF, YK), and the ``Bielefeld Young Researchers'' Fund through the Funding Line Postdocs: ``Career Bridge Doctorate\,--\,Postdoc'' (PT).

\end{document}